\documentclass{article}
\usepackage{enumerate}
\usepackage{amsmath,amsfonts,amssymb,amsthm}
\usepackage{color}
\usepackage{mathrsfs}
\usepackage{stmaryrd}
\usepackage{dsfont}
\usepackage{hyperref}

\providecommand{\R}{\mathbb{R}}

\providecommand{\N}{\mathbb{N}}

\providecommand{\eps}{\varepsilon}

\renewcommand{\leq}{\leqslant}
\renewcommand{\geq}{\geqslant}

\renewcommand{\div}{\operatorname{div}}
\newcommand{\curl}{\operatorname{curl}}

\newcommand{\CinT}{C_{\textnormal{in}}^T}

\newcommand{\dd}{\textnormal{d}}
\newcommand{\fin}{f_{\textnormal{in}}}
\newcommand{\Lip}{\textnormal{Lip}}
\newtheorem{Theorem}{Theorem}
\newtheorem{Definition}{Definition}
\newtheorem{Corollary}{Corollary}
\newtheorem{Proposition}{Proposition}
\newtheorem{Lemma}{Lemma}
\newtheorem{Remark}{Remark}
\begin{document}

\date{\today}
\title{A 2d spray model with gyroscopic effects.  }

\author{Ayman Moussa\footnote{CNRS, UMR 7598, Laboratoire Jacques-Louis Lions, F-75005, Paris, France}
\footnote{UPMC Univ Paris 06, UMR 7598, Laboratoire Jacques-Louis Lions, F-75005, Paris, France}, 
Franck Sueur\footnotemark[1],\footnotemark[2] 
}

\maketitle

\begin{abstract}
In this paper we introduce a PDE system which aims at describing the dynamics of a dispersed phase of particles moving into an incompressible perfect fluid, in two space dimensions. The system couples a Vlasov-type equation and an Euler-type equation: the fluid acts on the dispersed phase through a gyroscopic force whereas the latter contributes to the vorticity of the former.

 First we give a Dobrushin type derivation of the system as a mean-field limit of a PDE system which describes the dynamics of a finite number of massive pointwise particles moving into an incompressible perfect fluid.
  This last system is itself inferred from the paper \cite{smallbody}, where  the system for one massive pointwise particle was derived as the limit of the motion of a solid body  when the body shrinks to a point with fixed mass and  circulation.

Then we deal with the well-posedness issues including the existence of weak solutions.
Next we exhibit the  Hamiltonian structure of the system and 
finally, we study the behavior of the system in the limit where the mass of the particles vanishes.
\end{abstract}

\section{Introduction}

\subsection{The model at stake}

This paper is devoted to the following PDE system
\begin{eqnarray}
\label{VE1}
 \partial_t \omega + \div_x (\omega  u)  &=& 0 ,
\\   \label{VE2} 
 \partial_t  f + \div_x (f  \xi ) + \div_\xi (  f (\xi - u  )^{\perp}) &=& 0 ,
\end{eqnarray}
where 
\begin{eqnarray}
\label{VE3} u :=K  \lbrack \omega + \rho  \rbrack  \text{ and } \rho :=  \int_{\R^2  } f d\xi .
\end{eqnarray}
This system describes a dispersed phases of particles (or spray) moving into an incompressible perfect fluid, using a kinetic/fluid approach. The fluid is described \emph{via} macroscopic quantities : $\omega$ stands for its vorticity and $u$ for its velocity, both depending on the time $t\in\R_+$ and the space variable $x\in\R^2$. The spray is represented by a probability density function : $f$, depending on $t$, $x$ and $\xi \in \R^2$, the velocity of the particles. 
The notation $K$ in  \eqref{VE3} stands for the 
Biot-Savart operator associated to the full plane; it maps a reasonable scalar function $g$ to the vector field
\begin{eqnarray}
\label{BS}
K \lbrack g \rbrack (x) :=  \int_{ \R^{2}} H (x-y)  g (y) dy , 
\end{eqnarray}
where 
\begin{eqnarray}
\label{NBS}
H(x)  = \frac{1}{2\pi} \frac{x^{\perp}}{|x|^{2}} .
\end{eqnarray}
Here  the notation $x^\perp $ stands for $x^\perp =( -x_2 , x_1 )$, 
when $x=(x_1,x_2)$.
 Let us also observe that 
\begin{eqnarray*}
H(x)  = \nabla^\perp G (x)  .
\end{eqnarray*}
where $G$ is the $2$d Newtonian potential
\begin{eqnarray*}
  G (x) :=  \frac{1}{2 \pi} \ln (x) 
  \end{eqnarray*}
 and $ \nabla^\perp$ denotes the vector field $(-\partial_2 , \partial_1 )$.
  Moreover $K [g] $ is divergence-free and such that $\curl  K [g] = g$.
Eq.\eqref{VE3} defines the vector field $u$ which describes the fluid velocity. This vector field $u$ couples the Vlasov equation \eqref{VE1} and the Euler equation \eqref{VE2}: the fluid gives a lift acceleration to the particles (the $\div_\xi$ term in \eqref{VE2}), whereas the spray, as a whole, contributes to the vorticity of the former (equations $\eqref{VE3}$). Following the naming of \cite{oro}, let us stress that the system \eqref{VE1}-\eqref{VE3} describes the behavior of a thin spray: there is no interaction (collisions, coalescences) between particles thus the kinetic equation is linear in $f$ (given $u$) and furthermore the volume fraction occupied by the spray is neglected.

\subsection{A comparison with some other fluid/kinetic coupling}

Numerous examples of fluid/kinetic coupling have been studied since the seminal models introduced in \cite{caf-pap,wil}, but the coupling is usually made \emph{via} a drag force and not a lift force. Let us compare our system with some of the recents models introduced in the literature. An alternative way to describe the system \eqref{VE1}-\eqref{VE3} is to use the following velocity formulation:
\begin{eqnarray}
\label{VE1v}
 \partial_t u  + \div_x (u \otimes  u)  + \nabla p &=& (\rho u - j)^\perp  ,
 \\ \label{VE1vDIV} \div_x u  &=& 0 
\\   \label{VE2v} 
 \partial_t  f + \div_x (f  \xi ) + \div_\xi (  f (\xi - u  )^{\perp}) &=& 0 ,
\end{eqnarray}
where 
\begin{eqnarray}
\label{VE3v}  \rho :=  \int_{\R^2  } f d\xi   \text{ and } j := \int_{ \R^{2}}  f \xi d\xi .
\end{eqnarray}

Formally  one can obtain the previous formulation from  the system   \eqref{VE1}-\eqref{VE3} as follows:
one integrates the equation  \eqref{VE2} with respect to $\xi \in  \R^{2}$ to obtain the macroscopic conservation law:
\begin{eqnarray}
\label{VE4v}  \partial_t \rho + \div_{x} j  = 0.
\end{eqnarray}
Then one adds  the equation  \eqref{VE4v} to  the equation  \eqref{VE1} to get 
\begin{eqnarray}
\label{VE5v}  \partial_t (\omega + \rho ) + \div_{x} (j  +\omega u )= 0.
\end{eqnarray}
But, taking into account that $u$ is divergence free, 
\begin{eqnarray*}
 \partial_t (\omega + \rho ) + \div_{x} \big( (\omega  + \rho )u \big)=  \curl \big(  \partial_t u  + \div_x (u \otimes  u)  \big)  \text{ and }
 \  \div_{x} (j - \rho u )  =  - \curl \big( (\rho u - j)^\perp  \big) ,
\end{eqnarray*}
so that \eqref{VE5v} now reads:
\begin{eqnarray*}
  \curl \big(  \partial_t u  + \div_x (u \otimes  u)  -  (\rho u - j)^\perp  \big) = 0 ,
\end{eqnarray*}
what means that $ \partial_t u  + \div_x (u \otimes  u) -  (\rho u - j)^\perp$ is a gradient, thus yielding  \eqref{VE1v}.

The system $(7)-(9)$ of \cite{baranger} reads, with the present notations and for a homogeneous incompressible fluid, 
\begin{eqnarray}
\label{VE1vBD}
 \partial_t u  + \div_x (u \otimes  u)  + \nabla p &=& j - \rho u   ,
  \\ \label{VE1vDIVBD} 
  \div_x u  &=& 0,
\\   \label{VE2vBD} 
 \partial_t  f + \div_x (f  \xi ) + \div_\xi (  f (u - \xi )) &=& 0 ,
\end{eqnarray}
This can also be seen as the model $(1)-(3)$ of \cite{Carrillo} where the diffusion term of the kinetic equation has been dropped.

The model \eqref{VE1vBD}-\eqref{VE2vBD} involves the drag force $ (j - \rho u )$ whereas the model  \eqref{VE1v}-\eqref{VE2v} involves the gyroscopic force $-  (j - \rho u )^\perp$.
Let us mention that if several papers have dealt with spray models with drag effect, see for instance the introduction of \cite{Carrillo} for more references, it is, to our knowledge, the first paper where a spray with a gyroscopic effect is considered.

\subsection{Plan of the paper}

In the next section we recall the basics from optimal transportation theory that we will need in the sequel, in particular in Section \ref{MEF} where we will derive  the
system \eqref{VE1}-\eqref{VE3} as  a mean-field limit of a PDE system which describes the  dynamics of a finite number of massive pointwise particles moving into an incompressible perfect fluid. This last model originates from the paper \cite{smallbody}  and can be seen as an extension of the point vortices-wave system introduced by Marchioro and Pulvirenti to the case where the vortices have a non vanishing mass. 
We will obtain  \eqref{VE1}-\eqref{VE3} in the mean-field limit in a regularized setting by adapting the famous approach used by Dobrushin  \cite{Dobrushin} for a regularized Vlasov-Poisson system.

In  Section \ref{existence} we will give some existence results for the solutions of  \eqref{VE1}--\eqref{VE3} which include the case of quite irregular initial data.

In  Section \ref{Younik} we will give a uniqueness result for the solutions of  \eqref{VE1}--\eqref{VE3}  with $ \omega$ and $\rho$ bounded.
We will follow the analysis of the Vlasov-Poisson system performed by Loeper in \cite{Loeper} which makes use of optimal transportation.

In section \ref{PoissonStructure}
we exhibit a Hamiltonian structure of the equations 
by introducing a  Poisson structure on the  manifold $ \mathcal{P}$ of the pairs $(\omega,f)$ such that equations \eqref{VE1}--\eqref{VE3} yields some Hamiltonian ordinary differential equations for any  smooth functional of a solution $(\omega,f)$.

Finally in Section \ref{MassLess} we study what happens  when the mass of  the particles immersed into the fluid vanishes. We will see that the particles are then convected  by the fluid velocity so that the system degenerates into the usual  incompressible Euler equations.
We will follow a  strategy used by Brenier to deal with the gyrokinetic limit of the 
 Vlasov-Poisson equations.

\subsection{A few general notations}

In the whole paper, $\N$ will denote the set of nonnegative integers including $0$, whereas the notation $\N^*$ will stand for $\N^* := \N \setminus 0$. 
We will simply denote $| \cdot | $ the Euclidean norm in $\R^d$, where $d$ denotes a positive integer.\vspace{1mm}\\
A mapping $f$ from $\R^{d_1}$ to  $\R^{d_2}$ is said Lipschitz if there exists a constant $c>0$ such that for any $x,y$ in  $\R^d$, $| f(x) - f(y)  | \leq c | x-y | $. Then the smallest admissible constant $c$ defined a norm denoted $\| f  \|_{\text{Lip}}$ and the associated Banach space is noted $\textnormal{Lip}(\R^{d_1},\R^{d_2})$, or simply $\textnormal{Lip}(\R^{d_1})$ if $d_2=1$.
\vspace{1mm}\\
When no ambiguity arises on the time interval $[0,T]$ considered, we will denote by $L^p_{t}(L^q_x)$ and $L^p_t(L^q_{x,\xi})$ the spaces $L^p\big([0,T];L^q(\R^2)\big)$ and $L^p\big([0,T];L^q(\R^2\times\R^2)\big)$, for all $p,q\in[1,\infty]$. Similar notation for $\mathscr{C}^0_t(E)$ for continuous functions of $t$ having value in some functional space $E$. When no variable is added in index, $\|h\|_{L^p}$ denotes the $L^p$ norm of $h$ in all its variables, same thing for Sobolev norms. Thus for example, if $h$ is defined for $(t,x)\in[0,T]\times\R^2$, $\|h\|_{L^p}$ denotes its $L^p([0,T]\times\R^2)$ norm, while $\|h(t)\|_{L^p}$ denotes the $L^p(\R^2)$ norm of $h(t)$. Given an open subset $\Omega$ of  $\R^d$,  $\mathscr{D}(\Omega)$ denotes the space of test functions having compact support in $\Omega$, $\mathscr{D}'(\Omega)$ is hence the associated dual space of distributions.
\vspace{1mm}\\
We will use the same notation $C$ for some constants which may change from line to line. We will sometimes use $C_a$ to precise that the constant yet depends on $a$. 

\section{A few tools from optimal transportation theory}

We fix from now on a few notations from optimal transportation for which we refer to \cite{AGS,Villani}.
Here $d$ (it may be indexed), denotes a positive integer. It will be taken equal to $2$ or $4$ in the sequel, as the vorticity $\omega(t)$  depends on $x \in \R^2$ while the density  $f(t)$ depends on $(x,\xi)$ in $\R^2 \times \R^2$. Note also that $f$ being a probability density function, it is always nonnegative, whereas we do not specify  any sign condition for the fluid vorticity $\omega$.

\subsection{Measures}

We will  denote by  $\mathcal{M} ( \R^d) $ the set of signed measures and by $\mathcal{M}^+(\R^d)$  the set of finite measures on $ \R^d$. 
 We endow  $\mathcal{M} ( \R^d) $ with the standard narrow convergence, given by the duality with continuous and bounded functions, that is to say that a sequence $(\nu_n )_{n}$ in $\mathcal{M} ( \R^d) $ narrowly converges to $\nu$ in $\mathcal{M} ( \R^d) $ if for any $\phi$ in the space $\mathscr{C}_{b}^0 ( \R^d) $ of the  continuous and bounded functions over $\R^{d}$ (taking real values), 
\begin{align*}
\int_{\R^d } \phi \,d\nu_{n} \operatorname*{\longrightarrow}_{n  \rightarrow +\infty}  \int_{\R^d } \phi \,d\nu.
 \end{align*}
  For a measure $\nu$ in  $\mathcal{M} ( \R^d) $ we denote by $|\nu|$ its total variation measure, and by  $\nu^{+}$ and  $\nu^{-}$ its positive and negative parts, all three given by the Jordan-Hahn decomposition so that $\nu = \nu^{+} - \nu^{-}$, and $|\nu|=\nu^+ + \nu^-$. 
 
 We denote  $\mathcal{P} ( \R^d) $ the set of the probability measures  that is the subset of $\mathcal{M}^{+} ( \R^d)$ verifying $\nu(\R^d)=1$.
 
Given a measure $\nu$ in $\mathcal{M} (\R^{d_1}) $ and a map $\tau : \R^{d_1} \rightarrow  \R^{d_2} $  we define the pushforward measure $\tau_{\#}\nu\in\mathcal{M}(\R^{d_2})$ of $\nu$ by $\tau$  by 
\begin{eqnarray}
\label{PF}
 \tau_{\#}\nu  [B] :=   \nu [\tau^{-1}(B)] , 
\end{eqnarray}
for any Borel subset $B$ of $\R^{d_2}$.

Given a measure $\gamma$ over the product space $\R^d \times \R^d =\R^{2d} $ and the projections $\pi^1$ and $\pi^2$ on it factors, we define the first and second marginals of $\gamma$ as the measures $\pi^1_{\#} \gamma$ and $\pi^2_\#\gamma$. This means that for any Borel set $B$ of $\R^d$,
 \begin{eqnarray*}
\pi^1_{\#} \gamma [B] = \gamma  [B \times  \R^d] \text{ and } \pi^2 _{\#} \gamma  [B] = \gamma  [\R^d \times  B].
\end{eqnarray*}

 For any $p\in[1,\infty[$  we introduce the subspace of signed measures having a finite $p$-th moment
\begin{align*}
\mathcal{M}_p(\R^d) := \left\{\nu \in \mathcal{M}(\R^d)\,:\, \int_{\R^d}|x|^p d |\nu| < \infty\right\},
\end{align*}
and defined similarly $\mathcal{M}_p^+(\R^d)$ and $\mathcal{P}_p(\R^d)$.
 
 \subsection{Wasserstein distances for measures}

 We will make use of the Wasserstein distances. Let us first recall the standard definition for measures.
\begin{Definition}[] 
\label{Wasserstein}
 Let $p$ be a real satisfying $p\geqslant 1$,  $\nu_1$ and  $\nu_2$ be in $\mathcal{M}^+_{p} ( \R^d) $ two finite measures on $\R^d$ such as $\nu_1(\R^d) = \nu_2(\R^d)$. We define the Wasserstein distance $W_p ( \nu_1 ,\nu_2 )$ of order $p$ between  this two measures by 
\begin{eqnarray}
W_{p} ( \nu_1 ,\nu_2 ) := \left( \inf_{ \gamma} \int_{\R^d  \times  \R^d} \gamma (x,y)  | x-y  |^p dx dy \right)^{1/p}
\end{eqnarray}
where the infimum is taken over all $\gamma \in \mathcal{M}^+(\R^d \times \R^d)$ with marginals  $\nu_1$ and  $\nu_2$.
\end{Definition}
It can be shown that for real $p$ satisfying $p\geqslant 1$, for any positive real number $\alpha >0$,  $W_p $ satisfies all the axioms of a metric on the subset of $\mathcal{M}^{+}_p ( \R^d)$ of measures verifying $\nu(\R^d)=\alpha$.
 Furthermore, convergence with respect to  $W_p$ is equivalent to the narrow convergence of measures plus convergence of $p$th moments.

 \subsection{Wasserstein distances for signed measures}

For signed measures, there are several ways to define some Wasserstein distances. Let us refer here to \cite{AMS}.
In the sequel we will only need to consider some cases where the  positive  parts have the same  total mass, and the negative parts as well, this motivates the following definition
\begin{Definition}
We will say that two signed measures $\nu_1$ and $\nu_2$ on $\R^d$ are compatible whenever $\nu_1^\pm(\R^d)=\nu_2^\pm(\R^d)$. More generally, we will say that a family of signed measures is compatible if any pair of its elements is compatible.
\end{Definition} 
Yet we will distinguish two cases depending on whether $p$ is equal to $1$ or not. Let us start with the second case.
\begin{Definition}[] 
\label{WassersteinSigne}
 Let $p\in]1,\infty[$. Given two compatible elements $\nu_1$, $\nu_2$  of $\mathcal{M}_{p}  ( \R^d) $, we define the Wasserstein distance $W_p ( \nu_1 ,\nu_2 )$ of order $p$ between  these two measures by 
\begin{eqnarray}
W_{p} ( \nu_1 ,\nu_2 ) := \Big(  W_{p} ( \nu^+_1 ,\nu^+_2 )^p +  W_{p} ( \nu^-_1 ,\nu^-_2 )^p \Big)^\frac{1}{p} .
\end{eqnarray}
\end{Definition}
Of course we recover Definition \ref{Wasserstein} in the case of positive measures.
We will use Definition \ref{WassersteinSigne} with $p=2$ in Section \ref{Younik}.

In the case where  $p=1$, we extend Definition \ref{Wasserstein} by setting 
\begin{Definition}[] 
\label{WassersteinSigne1}
 Given two compatible elements $\nu_1$, $\nu_2$ of $\mathcal{M}_{1}  ( \R^d) $, we define the Wasserstein distance $W_1 ( \nu_1 ,\nu_2 )$ between  these two measures by 
\begin{eqnarray}
W_{1} ( \nu_1 ,\nu_2 ) := W_{1} (  \nu_1^{+} +  \nu_2^{-}  ,  \nu_1^{-} + \nu_2^{+} ) .
\end{eqnarray}
\end{Definition}
Such a definition allows us to keep the well-known Kantorovitch duality :
\begin{Proposition}
\label{Kanto}
Let $\nu_1$ and $\nu_2$ be two compatible elements of $\mathcal{M}_{1}  ( \R^d) $.
Then 
\begin{eqnarray}
\label{NormS}
 W_{1} ( \nu_1 ,\nu_2 ) = \sup_{\phi }  \,   \int_{\R^d  } \phi  \,d( \nu_1 -  \nu_2),
\end{eqnarray}
where the supremum is taken over the unit ball of $\textnormal{Lip}(\R^d)$.
\end{Proposition}
\begin{proof}
This formula is well-known for positive measures, see for instance \cite{AGS,Villani}. Moreover from Definition  \ref{WassersteinSigne1}, we infer that 
\begin{eqnarray*}
W_{1} ( \nu_1 ,\nu_2 ) &= & \sup_{\phi}   \,   \int_{\R^d  } \phi  \,d( ( \nu_1^{+} +  \nu_2^{-}  ) -  (  \nu_1^{-} + \nu_2^{+})) ,
\\  &=& \sup_{\phi}  \,   \int_{\R^d  } \phi\,  d( \nu_1  -    \nu_2 ) .
\end{eqnarray*}
\end{proof}
Let us give a few properties of the distance $W_{1} $. The first one is that it is really a distance: this comes easily from Proposition \ref{Kanto}.
For instance the positivity and the symmetry are quite obvious whereas the definiteness comes from Riesz theorem and the triangle inequality can be checked as follows: given three compatible signed measures $\nu_1 ,\nu_2  , \nu_3$ in  $\mathcal{M} ( \R^d) $, one has (the supremum is always taken over the unit ball of $\text{Lip}(\R^d)$)
\begin{eqnarray*}
W_{1} ( \nu_1 ,\nu_3 )  &=& \sup_{ \phi   }  \,   \int_{\R^d  } \phi\,  d( \nu_1 -  \nu_3) ,
\\  &=& \sup_{ \phi }  \,   \int_{\R^d  } \phi\,  d( \nu_1 -  \nu_2 + \nu_2 - \nu_3) ,
\\ &\leqslant&  \sup_{ \phi }  \,   \int_{\R^d  } \phi\,  d( \nu_1 -  \nu_2 )+  \sup_{ \phi }  \,   \int_{\R^d  } \phi\,  d \nu_2 - \nu_3) ,
\\  && = W_{1} ( \nu_1 ,\nu_2 )  + W_{1} ( \nu_2 ,\nu_3 )  .
\end{eqnarray*}
Note also that  the distance $W_{1} $ satisfies the following.
\begin{Lemma}
\label{4}
If $\nu_1 ,\nu_2$  and $\nu_3, \nu_4$ are two pairs of compatible elements of  $\mathcal{M}_{1}  ( \R^d) $, then the two measures $\nu_1+\nu_3$ and $\nu_2+\nu_4$ are compatible and we have
\begin{eqnarray*}
W_{1} ( \nu_1 +  \nu_3 ,  \nu_2 +  \nu_4)   \leqslant W_{1} ( \nu_1  ,  \nu_2 ) +  W_{1} (  \nu_3 ,    \nu_4)  ,
\end{eqnarray*}
with equality in the particular case where $ \nu_3 =  \nu_4$, what amounts to saying that the distance $W_{1} $ is translation invariant.
\end{Lemma}
\begin{proof}
The compatibility is straightforward. Thanks to Proposition \ref{Kanto} we have  (supremum is taken over the unit ball of $\text{Lip}(\R^d)$)
\begin{eqnarray*}
W_{1} ( \nu_1 +  \nu_3 ,  \nu_2 +  \nu_4) &=& \sup_{ \phi }  \,   \int_{\R^d  } \phi  \,d ( \nu_1 +  \nu_3 -  \nu_2 -  \nu_4) ,
\\ &=& \sup_{ \phi }  \,  \Big(  \int_{\R^d  } \phi\,  d ( \nu_1  -  \nu_2 )+  \int_{\R^d  } \phi  \,d (\nu_3  -  \nu_4)   \Big)
\\ &\leqslant& \sup_{  \phi }  \,  \int_{\R^d  } \phi\,  d ( \nu_1  -  \nu_2 )+ \sup_{ \phi }  \,  \int_{\R^d  } \phi\,  d (\nu_3  -  \nu_4)  ,
\\   &&= W_{1} ( \nu_1  ,  \nu_2 ) +  W_{1} (  \nu_3 ,    \nu_4)  .
\end{eqnarray*}
The case  where $ \nu_3 =  \nu_4$ is straightforward.
\end{proof}
\begin{Remark}
Lemma \ref{4} can also be seen as a consequence of the fact that the distance $W_{1}$ can be associated to a norm, see \cite{Hanin}.
\end{Remark}
\begin{Remark}
Applying Lemma \ref{4} with $( \nu_1^{+} ,\nu_2^{+} , \nu_2^{-} ,\nu_1^{-} )$ instead of $(\nu_1 ,\nu_2  , \nu_3 ,  \nu_4 )$ we have that 
$ W_{1} ( \nu_1 ,\nu_2 )   $ is less than $   W_{1} ( \nu_1^{+} ,\nu_2^{+} ) + W_{1} ( \nu_1^{-} ,\nu_2^{-} ) $,
which is the result given by the formula in Definition \ref{WassersteinSigne} for $p=1$.
\end{Remark}
\begin{Lemma}
\label{CompoLip}
Consider two compatible elements $\nu_1$, $\nu_2$ of $\mathcal{M}_{1}  ( \R^{d_1}) $. Let $\tau : \R^{d_1} \rightarrow  \R^{d_2} $ be a Lipschitz map. 
Then 
\begin{eqnarray*}
 W_{1} (  \tau_{\#} \nu_1 , \tau_{\#}  \nu_2 ) \leqslant \| \tau  \|_{\textnormal{Lip}}  W_{1} (\nu_1 ,\nu_2 )  .
\end{eqnarray*}
\end{Lemma}
\begin{proof}
Thanks to \eqref{NormS} we have 
\begin{eqnarray*}
 W_{1} (\tau_{\#} \nu_1 , \tau_{\#}  \nu_2 ) &=& \sup_{ \phi  } \,    \int_{\R^{d_2}  } \phi  \,d (\tau_{\#} \nu_1 - \tau_{\#}\nu_2 ) ,
 \\  &=& \sup_{ \phi  }  \,    \int_{\R^{d_2}  } \phi \circ \tau \,  d (\nu_1 - \nu_2 ) ,
 \\  &=&\|   \tau \|_{\text{Lip}} \,    \sup_{\phi }  \,    \int_{\R^{d_2}  } \frac{\phi \circ \tau}{\|   \tau \|_{\text{Lip}}}   \,d (\nu_1 - \nu_2 ) ,
 \\  &\leqslant  &  \|   \tau \|_{\text{Lip}} \,  \sup_{\psi }  \,   \int_{\R^d  } \psi\,  d   (\nu_1 - \nu_2 )  ,
  \\  &&  = \| \tau  \|_{\text{Lip}}   \,  W_{1} (\nu_1 ,\nu_2 ),
\end{eqnarray*}
where the index $\phi$ denotes a supremum over the unit ball of $\text{Lip}(\R^{d_2})$ and $\psi$ a supremum over the unit ball of $\text{Lip}(\R^{d_1})$.
\end{proof}
In the particular case where the map $\tau$ of the previous lemma is one of the projection maps defined above, we get the following result
\begin{Lemma}
\label{margi}
 Consider two compatible elements $\nu_1$, $\nu_2$ of $\mathcal{M}_{1} ( \R^d \times \R^d) $.
We have, for $i=1,2$, 
\begin{eqnarray*}
W_{1} ( \pi^i_{\#} \nu_{1} , \pi^i_{\#} \nu_{2} )  \leqslant W_{1} (  \nu_{1} ,\nu_{2} ) .
\end{eqnarray*}
\end{Lemma}

 \subsection{Wasserstein distances for vector measures}
 
As we will deal with the coupled system \eqref{VE1}-\eqref{VE2} for which there are two unknowns: $\omega$ and $f$, we will use the following extension of Definition \ref{WassersteinSigne}: 
\begin{Definition}[] 
\label{WassersteinSigneDistance}
Let  $d_{1}$ and  $d_{2}$ be two positive integers. Consider two pairs $\nu_1$, $\nu_2$  and $\sigma_1$, $\sigma_2$ of compatible measures (respectively in $\mathcal{M}_{1} ( \R^{d_{1}}) $ and $\mathcal{M}_{1} ( \R^{d_{2}}) $). We introduce the couples $\mu_i:=(\nu_i,\sigma_i)\in\mathcal{M}_1(\R^{d_1})\times\mathcal{M}_1(\R^{d_2})$, $i=1,2$ and define the associated Wasserstein distance $W_1 ( \mu_1 ,\mu_2 )$ between $\mu_1$ and $\mu_2$  by 
\begin{eqnarray*}
W_{1} ( \mu_1 ,\mu_2) := W_{1} ( \nu_1  , \nu_2 ) + W_{1} ( \sigma_1  , \sigma_2 )  .
\end{eqnarray*}
\end{Definition}

\section{Derivation of the system as a mean-field model}
\label{MEF}

\subsection{Massive vortex-wave system}

In \cite{smallbody} we derive the following system for the motion of one massive vortex of mass $m >0$ and circulation  $\gamma \in \R$ in a perfect incompressible flow. Recalling the notation
\begin{align*}
H(x) := \frac{1}{2\pi} \frac{x^\perp}{|x|^2},
\end{align*}
the system reads 
\begin{gather}
\label{EulerPoint} 
\partial_{t} \omega  +  \div_x (  \omega  u  ) = 0 , \\
\label{EulerU} 
u(t,x) := K \lbrack \omega \rbrack (t,x) +  \gamma H(x-h(t)) \\
\label{PointEuler}
m h''(t) = \gamma \Big(h'(t) -  K \lbrack \omega \rbrack (t,h(t))\Big)^\perp, \\
\label{EP0}
\omega |_{t= 0}=  \omega_0 ,\ h(0) = h_0, \ h' (0) = h_1 , \\
\end{gather}
This system was obtained by considering  the motion of a solid body in a two dimensional incompressible perfect fluid, when the body shrinks to a pointwise particle at the  position $h(t)$ with a fixed mass $m>0$ and a fixed circulation $\gamma \in \R$ around the body.
Equation (\ref{EulerPoint}) describes the evolution of the vorticity $\omega $ of the fluid: it is transported by the velocity $u$ obtained by the usual Biot-Savart law in the plane, but from a vorticity which is the sum of the fluid vorticity and of a point vortex placed at  $h(t)$  with a strength equal to the circulation $\gamma$. Observe for instance that the  velocity $u$ is divergence free and can be written as 
$u =  K \lbrack \omega +   \gamma    \delta_{h(t)}  \rbrack $. \par
Equation (\ref{PointEuler}) means that the shrunk body is accelerated by a  force similar to the Kutta-Joukowski lift of the irrotational theory: the shrunk body experiments a lift which is proportional to the circulation $\gamma$ and to the difference between the solid velocity and the virtual fluid velocity  obtained by the Biot-Savart law in the plane from the fluid vorticity, up to a rotation of a $\pi/2$ angle. Let us refer here to the textbooks of Childress \cite{Childress} or Marchioro and Pulvirenti \cite{MP} for a discussion of the Kutta-Joukowski force. See also Grotta-Ragazzo, Koiller and Oliva \cite{ragazzo}, where they consider a similar system of a point mass embedded in an irrotational fluid and driven by the Kutta-Joukowski force. \par

The system (\ref{EulerPoint})-(\ref{EP0})  can also be seen as a variant of the vortex-wave system introduced by Marchioro and Pulvirenti cf. for instance Th. $6.2$ of \cite{MP} and \cite{MP2}, and recently studied in \cite{criline} and in \cite{cb}.
Actually the  system (\ref{EulerPoint})-(\ref{EP0})   reduces to the vortex-wave system when the mass $m$ is set to $0$.

\subsection{$N$ particles }

Let us now generalize the previous system to the case of  $N$ pointwise particles of mass $m_i$, of circulation $\gamma_i $ and of position $h_i (t)$, for $i=1,...,N$, moving into a perfect and incompressible planar fluid:
\begin{eqnarray}
\label{N1}  \partial_t \omega +  \div_x (   \omega u)  = 0 ,
   \\  \label{N1u} u(t,x) = K \lbrack \omega \rbrack(t,x) + \sum_{j = 1}^N \gamma_j H (x  - h_j(t) )  , 
\\  \label{N2} m_i  h_i '' (t) = \gamma_i  \Big( h_i '(t) -   v_{i} (t, h_i (t)  ) \Big)^\perp , 
 \\  \label{N2u} v_{i}(t,x) = K \lbrack \omega \rbrack (t,x) +  \sum_{j \neq i} \gamma_j H ( x - h_j (t) ) ,
\\	\label{N5}	 \omega |_{t= 0} =  \omega_0 ,\   h_i  (0)    = h_{i, 0} ,\	 h'_i  (0)   = h_{i, 1} .
\end{eqnarray}
Let us observe that in \eqref{N2u} the self-interaction is omitted since the index in the sum runs only over $j \neq i$. 
The derivation of the system \eqref{N1}-\eqref{N5} from the motion of $N$  solid bodies  in a two dimensional incompressible perfect fluid, when the bodies shrink to  pointwise particles, is the object of a paper in preparation.

Once again if one sets all the masses $m_i$ equal to $0$ one recovers the vortex-wave system  of Marchioro and Pulvirenti.

\subsection{Mean-field limit }
\label{M-fl}
We want to study the mean-field limit of the system \eqref{N1}-\eqref{N5}, that is the limit system obtained by the empirical measure
\begin{eqnarray*}
f_N (t) := \frac{1}{N} \sum_{i=1}^{N } \delta_{(h_i (t) , h'_i (t))} 
\end{eqnarray*}
when $N$ goes to infinity, with an appropriate scaling of the amplitudes.
We therefore consider now the solutions of 
\begin{eqnarray*}
 \partial_t \omega +  \div_x (   \omega u) = 0 , 
 \\   u(t,x) = K \lbrack \omega \rbrack(t,x) + \frac{1}{N} \sum_{j = 1}^N  H (x  - h_j (t) )  , 
\\    h_i '' (t) =   \Big( h_i '(t) -  v_{i} (t, h_i (t)  )  \Big)^\perp   , 
 \\   v_{i}(t,x) = K \lbrack \omega \rbrack(t,x)+ \frac{1}{N}  \sum_{j \neq i} H ( x - h_j (t) ) ,
\\	 \omega |_{t= 0} =  \omega_0 ,\   h_i  (0)    = h_{i, 0} ,\	 h'_i  (0)   = h_{i, 1} .
\end{eqnarray*}
\subsection{A regularized version of the system }
Such an issue is quite similar to the one which consists to obtain the Vlasov-Poisson system as a mean-field limit of  Newton's equations for charged particles.
This last problem is still open at the time of writing,%
\footnote{In the case of the Vlasov-Poisson system, Hauray and Jabin have recently succeed to improve the previously known results about the mean-field limit to some cases where the interaction kernel can be singular cf.  \cite{HJ}. Yet their approach does not cover the case of the Newtonian kernel appearing in \eqref{VE3}.} but some results are available in the  simpler setting where $H$ is assumed to be $W^{1,\infty}$, in the spirit of the famous paper   \cite{Dobrushin} by Dobrushin  (let us also mention here   Braun and Hepp \cite{braun} and  Neunzert \cite{neunzert}).
We will therefore consider the equations \eqref{VE1}-\eqref{VE2} where $u$ is given by 
\begin{eqnarray}
\label{VE3R} u := \widetilde{K}  \lbrack \omega + \rho  \rbrack  \text{ and } \rho :=  \int_{\R^2  } f d\xi ,
\end{eqnarray}
where $ \widetilde{K}$ is defined by
\begin{eqnarray*}
 \widetilde{K} \lbrack g \rbrack (x) :=  \int_{ \R^{2}}  \widetilde{H} (x-y)  g (y) dy , 
\end{eqnarray*}
with $ \widetilde{H}$ is in $W^{1,\infty} ( \R^2 )$ and satisfies $ \widetilde{H}(0)=0$.
\begin{Theorem} 
\label{MFR}
Assume $\widetilde{H}\in W^{1,\infty}(\R^2)$, satisfying $\widetilde{H}(0)=0$ and $\widetilde{K}$ defined as above. We have the following.
\begin{enumerate}[(a)]
\item Assume that $( \omega_{0} , f_{0}) $ is in $\mathcal{M}_1 ( \R^2 ) \times \mathcal{P}_1 ( \R^2 \times  \R^2 ) $.  
Then there exists only one  weakly continuous curve $(\omega_t , f_t )$ in $\mathscr{C}^0([0,\infty) ;  \mathcal{M}_1 ( \R^2) \times \mathcal{P}_1 ( \R^2 \times  \R^2 ) )$ solution of the system   \eqref{VE1}-\eqref{VE2}-\eqref{VE3R} with $( \omega_{0} , f_{0}) $ as initial data. 
\item Moreover we have the following stability property.
Consider two solutions  $  \mu_1 := ( \omega_1 , f_1) $ and $\mu_2 :=  ( \omega_2 , f_2) $  of the system   \eqref{VE1}-\eqref{VE2}-\eqref{VE3R} associated to two initial data  $ \mu_{0}^1  := ( \omega_{0}^1 , f_{0}^1) $ and $ \mu_{0}^2 := ( \omega_{0}^2 , f_{0}^2) $ in $\mathcal{M}_1 ( \R^2 ) \times \mathcal{P}_1 ( \R^2 \times  \R^2 ) $. Then, for any $t\geq 0$,
\begin{eqnarray}
\label{stab}
W_{1} (  \mu_1 (t) ,\mu_2 (t) ) \leq e^{2Ct} \, W_{1} (  \mu_{0}^1 ,\mu_{0}^2 ) ,
\end{eqnarray}
where $C >0$ depends only on $ \|  \widetilde{H}  \|_{\textnormal{Lip}}$ and $|\omega_0|(\R^2)$.
\item  Finally if we assume that $( \omega_{0} , f_{0}) $ is also in $\Lip ( \R^2) \times \Lip ( \R^2  \times  \R^2 ) $, then the corresponding solution $(\omega_t , f_t )$ is in $L^\infty_{\textnormal{loc}} ([0,\infty) ;  \Lip ( \R^2  ) \times \Lip ( \R^2 \times  \R^2)  )$
\end{enumerate}
\end{Theorem}
Above $W_{1}$ is the Wasserstein distance defined in Definition \ref{WassersteinSigneDistance}.
In the proof of Theorem \ref{MFR} we will use the following notation: when $\nu$ is in  $\mathcal{M}_1 ( \R^2)$, we denote  $U  \lbrack \nu \rbrack $ the vector field defined for 
$  (x,\xi ) \in \R^2 \times  \R^2$ by 
\begin{equation*}
U   \lbrack \nu \rbrack  (x,\xi ) := \Big(\xi , (\xi - \widetilde{K} \lbrack \nu \rbrack (x) )^{\perp} \Big)  \in \R^2 \times  \R^2 . 
\end{equation*}
We will use the two following lemmata. 
\begin{Lemma}
\label{LipLip}
Let  $\nu$ be in  $\mathcal{M}_1 ( \R^2)$. 
Then  the vector fields  $ \widetilde{K} \lbrack \nu \rbrack $ and $U  \lbrack \nu \rbrack $  are Lipschitz and 
\begin{eqnarray*}
\|  \widetilde{K} \lbrack \nu \rbrack  \|_{\Lip} \leqslant |\nu|(\R^2)  \|  \widetilde{H}  \|_{\Lip}    \text{ and } \| U  \lbrack \nu \rbrack   \|_{\Lip} \leqslant   |\nu|(\R^2)\max ( \|  \widetilde{H}  \|_{\Lip } ,1)  .
\end{eqnarray*}
Moreover $ \widetilde{K} \lbrack \nu \rbrack $ is uniformly bounded
\begin{eqnarray}
\label{bdkernel}
\|  \widetilde{K} \lbrack \nu \rbrack  \|_{L^\infty} \leqslant  \,   |\nu|(\R^2)  \|  \widetilde{H}  \|_{L^\infty}  .
\end{eqnarray}
\end{Lemma}
\begin{proof}
This follows easily from convolution properties, referring to the definition of  $\widetilde{K} \lbrack \cdot \rbrack$ above.
\end{proof}
\begin{Lemma}
\label{compareU}
Let  $\nu_1$ and  $\nu_2$ be in  $\mathcal{M}_1 ( \R^2)$.
Then
\begin{eqnarray*}
\|  U \lbrack  \nu_1  \rbrack  - U  \lbrack  \nu_2  \rbrack   \|_{L^{\infty} } =  \|  \widetilde{K} \lbrack \nu_{1} \rbrack -  \widetilde{K} \lbrack \nu_{2} \rbrack  \|_{L^{\infty}}
 \leqslant  \|  \widetilde{H}  \|_{\Lip} \, W_{1 } (\nu_1 ,\nu_2 )  .
\end{eqnarray*}
\end{Lemma}
\begin{proof}
We have 
\begin{eqnarray*}
  U \lbrack  \nu_1  \rbrack - U \lbrack  \nu_2  \rbrack  =   \Big(0, - (  \widetilde{K} \lbrack \nu_{1}-  \nu_{2} \rbrack (x) )^{\perp} \Big) ,
\end{eqnarray*}
so that 
\begin{eqnarray*}
|  U \lbrack  \nu_1  \rbrack - U \lbrack  \nu_2  \rbrack|
 &=&  |  \widetilde{K} \lbrack \nu_{1}-  \nu_{2} \rbrack |
\\ &=&  |\widetilde{H}\star (\nu_1-\nu_2)|
\\ &=&  \|  \widetilde{H}  \|_{\Lip} \left|\frac{\widetilde{H}}{\|\widetilde{H}\|_\Lip}\star(\nu_1-\nu_2)\right|.
\end{eqnarray*}
It then remains to use Proposition  \ref{Kanto} to conclude the proof of the Lemma.
\end{proof}
\begin{proof}[Proof of Theorem \ref{MFR}]
Denote $\mu_0:=( \omega_{0} , f_{0}) \in\mathcal{M}_1 ( \R^2) \times \mathcal{P}_1 ( \R^2  \times  \R^2) $.

Consider the subset $\mathcal{F}_0\subset\mathcal{M}_1(\R^2)\times\mathcal{P}_1(\R^2\times\R^2)$ of elements $(\omega,f)$ such as $\omega$ is compatible with $\omega_0$. Then, for  $T>0$, $\mathcal{F}:=\mathscr{C}^0([0,T];\mathcal{F}_0)$ endowed with 
\begin{eqnarray} 
\label{DisT}
 \mathcal{W}_1 ( \mu_{1} ,  \mu_{2} ) := \sup_{t\in [0,T]} W_{1} ( \mu_{1} (t) ,  \mu_{2}  (t))
\end{eqnarray}
 is a complete metric space. 
 We first prove the local existence and uniqueness of solutions to the equations \eqref{VE1}-\eqref{VE2}-\eqref{VE3R} by using the Picard-Banach theorem  on $\mathcal{F}$  with $T>0$ small enough for the mapping $\mathcal{T}_{\mu_0}$ we are now going to describe to be a contraction for the previous distance.

 The mapping $\mathcal{T}_{\mu_0}$ is defined as follows. Denote by $\pi^1:\R^2\times\R^2\rightarrow\R^2$ the projection on the first variable. Given  $\mu_0 := ( \omega_{0} , f_{0}) $ in $ \mathcal{F}_0 $ and $\mu := ( \omega , f) $ in $ \mathcal{F} $, we define $\rho  \lbrack  f \rbrack(t) :=  \pi^1_\# f(t) = \langle f(t),1\rangle_{\xi} $, so that $\omega  + \rho  \lbrack  f \rbrack $ is in $\mathscr{C}^0([0,T];\mathcal{M}_{1} ( \R^2) )$. 
Then thanks to Lemma \ref{LipLip},  the vector fields  $ \widetilde{K} \lbrack  \omega  + \rho \lbrack  f \rbrack  \rbrack $ and  $U  \lbrack  \omega  + \rho \lbrack  f \rbrack  \rbrack$ are Lipschitz, uniformly in time; and we can therefore define 
\begin{itemize}
\item[$\bullet$] a unique continuous flow map  $ \phi^\mu $ defined for $(t,x) \in [0,T] \times \R^2$, with values in $ \R^2$, given by
  \begin{eqnarray}
  \label{Cara0}
\phi_t^\mu(x):=   x + \int^{t}_0  \widetilde{K}   \big[  \omega(s)  + \rho \lbrack  f \rbrack(s) \big] (\phi_s^\mu (x))  ds ,
   \end{eqnarray}
\item[$\bullet$] a unique continuous flow map  $ \Sigma^\mu $  defined for $(t,x,\xi) \in [0,T] \times \R^2 \times \R^2$, with values in $ \R^2 \times \R^2$, given by
 \begin{eqnarray}
   \label{Cara2} \Sigma_t^\mu(x,\xi):=  (x,\xi ) +  \int^{t}_0  U  \big[  \omega(s)  + \rho \lbrack  f \rbrack(s)\big] ( \Sigma_s^\mu (x,\xi) ) ds.
  \end{eqnarray}
\end{itemize}
Moreover, for all $t\in[0,T]$, $\mu_t^\mu$ and $\Sigma_t^\mu$ are Lipschitz and we have the following estimate
\begin{eqnarray}
\label{Lip0}
\|    \phi_t^\mu \|_{\Lip} + \|    \Sigma_t^\mu \|_{\Lip}& \leqslant& \exp\left\{\int_0^{t} \max ( \|  \widetilde{H}  \|_{\Lip} ,1) \,  \big[|\omega(s)|(\R^2)+1\big] ds\right\} \\
\nonumber && =\exp\left\{\int_0^{t} \max ( \|  \widetilde{H}  \|_{\Lip} ,1) \,  \big[|\omega_0|(\R^2)+1\big] ds\right\}.
\end{eqnarray}
We then define $\mathcal{T}_{\mu_0} \lbrack \mu  \rbrack$ as a function defined on $[0,T]$, with value in $\mathcal{M}(\R^2)\times\mathcal{M}(\R^2)$ by 
\begin{eqnarray*}
\mathcal{T}_{\mu_0} \lbrack \mu  \rbrack (t)
 := (  \omega^\mu(t),f^\mu(t)) = ({\phi_t^\mu}_{\hspace{-0.7mm}\#}\omega_0,,{\Sigma_t^\mu}_{\hspace{-0.6mm}\#} f_0)
 \end{eqnarray*}
Actually, $\mathcal{T}_{\mu_0}$ takes its values in $\mathcal{F}$.\\
 Indeed, since pushforward transformations conserve the mass and transport the Jordan-Hahn decomposition, for any $t \in [0,T]$, $(\omega^\mu(t),f^\mu(t))\in\mathcal{M}(\R^2)\times\mathcal{P}(\R^2)$, $\omega^\mu(t)$ being compatible with $\omega_0$.  Moreover using  \eqref{bdkernel} we have that for any $t \in [0,T]$, 
$\|   \widetilde{K} \lbrack  \omega  + \rho \lbrack  f \rbrack  \rbrack    (t,\cdot)  \|_{L^\infty  (\R^{2})} \leqslant  \alpha$ with $\alpha := \|  \widetilde{H}  \|_{L^\infty (\R^{2})} \, (1 +  |\omega_0|(\R^2))  $, so that we see from \eqref{Cara0}-\eqref{Cara2} that 
 $ |    \phi^\mu (t,x) | \leq |x| + \alpha t$ and $| \Sigma^\mu  (t,x,\xi) |  \leq ( |x|+|\xi| ) + t  ( \alpha +2|\xi| ) $. Since $\omega_0$ and $f_0$ are finite measures having finite first moments, the previous inequalities insure that so do $\omega^\mu(t)$ and $f^\mu(t)$, for all $t\in[0,T]$. Finally in order to prove that $\mathcal{T}_{\mu_0} \lbrack \mu  \rbrack$ is   in $\mathcal{F}$ it only remains to stress that the continuity in time of $  \phi^\mu  $ and $ f^\mu  $ yields the continuity in time of $\mathcal{T}_{\mu_0} \lbrack \mu  \rbrack$.

Let us now observe that  a fixed point of the mapping  $\mathcal{T}_{\mu_0}$ corresponds to a weak solution of the equations \eqref{VE1}-\eqref{VE2}-\eqref{VE3R} with $\mu_0 = ( \omega_{0} , f_{0}) $ as initial data.
Now let us see that  $\mathcal{T}_{\mu_0}$  is a contraction for $T$ small enough; consider  $\mu_{1} := ( \omega_{1} , f_{1}) $ and $\mu_{2} := ( \omega_{2} , f_{2}) $  in $ \mathcal{F}$. 

Let us define, for $t$ in $[0,T]$,
\begin{eqnarray*}
\lambda (t) := \int_{\R^{2 }}  |  \phi_t^{\mu_{1}}  -  \phi_t^{\mu_{2}}  |   d  |\omega_{0}|
\text{ and } \widetilde{\lambda} (t) := \int_{\R^{2 } \times \R^{2 }}  | \Sigma_t^{\mu_{1}}  -   \Sigma_t^{\mu_{2}}  | df_{0} .
\end{eqnarray*}
We have, from \eqref{Cara0}
\begin{eqnarray*}
\lambda' (t)  &\leqslant& \int_{\R^{2 }}  \left|   \widetilde{K}   \big[  \omega_{1}(t)  + \rho \lbrack  f_{1} \rbrack(t)\big] \circ \phi_t^{\mu_{1}}  -   \widetilde{K}   \big[  \omega_{2}(t)  + \rho \lbrack  f_{2} \rbrack(t)\big] \circ \phi_t^{\mu_{2}} \right|   d  |\omega_{0}|,
\end{eqnarray*}
thanks to \eqref{Cara0}. Now we use the triangle inequality to get $\lambda' (t) \leqslant \lambda_a (t) +  \lambda_b (t)$ with 
\begin{eqnarray*}
 \lambda_a (t)  &:=& \int_{\R^{2 }}  \left|   \widetilde{K}   \big[  \omega_{1}(t)  + \rho \lbrack  f_{1} \rbrack(t)\big] \circ \phi_t^{\mu_{1}} -   \widetilde{K}   \big[  \omega_{2}(t)  + \rho \lbrack  f_{2} \rbrack(t)\big] \circ \phi_t^{\mu_{1}} \right|   d  |\omega_{0}|  ,
 \\
 \lambda_b (t)  &:=&  \int_{\R^{2 }}  \left|   \widetilde{K}   \big[  \omega_{2}(t)  + \rho \lbrack  f_{2} \rbrack(t)\big] \circ \phi_t^{\mu_{1}} -   \widetilde{K}   \big[  \omega_{2}(t)  + \rho \lbrack  f_{2} \rbrack(t)\big] \circ \phi_t^{\mu_{2}} \right|   d  |\omega_{0}|  .
\end{eqnarray*}
We have, for any $t$ in $[0,T]$,
\begin{eqnarray*}
{W}_1 ( \omega_{1}  (t) + \rho \lbrack  f_{1} \rbrack (t) ,  \omega_{2}  (t) + \rho \lbrack  f_{2} \rbrack  (t))
&\leqslant& {W}_1 ( \omega_{1}  (t) ,  \omega_{2}  (t) ) + {W}_1 ( \rho \lbrack  f_{1} \rbrack (t) ,  \rho \lbrack  f_{2} \rbrack  (t)) \\
&\leqslant&  {W}_1 ( \mu_{1} (t) , \mu_{2} (t) ),
\end{eqnarray*}
where we used Lemma \ref{4} for the first inequality and  Lemma \ref{margi} for the second one, since $\rho\lbrack f_i\rbrack(t)=\pi^1_\# f_i(t)$, for $i=1,2$.
Thanks to Lemma \ref{compareU} we therefore obtain that
\begin{eqnarray*}
 \lambda_a (t)  &:=& \int_{\R^{2 }}  \left|   \widetilde{K}   \big[  \omega_{1}(t)  + \rho \lbrack  f_{1} \rbrack(t)\big]  -   \widetilde{K}   \big[  \omega_{2}(t)  + \rho \lbrack  f_{2} \rbrack(t)\big]  \right|   d  {\phi_t^{\mu_1}}_{\hspace{-1.8mm}\#}|\omega_{0}|  ,\\
\\  &\leqslant&  \|  \widetilde{H}  \|_{\Lip} \,   {W}_1 ( \mu_{1} (t) , \mu_{2} (t) ) \   {\phi_t^{\mu_1}}_{\hspace{-1.8mm}\#}|\omega_{0}|(\R^2)
\\ &=&\|  \widetilde{H}  \|_{\Lip} \,   {W}_1 ( \mu_{1} (t) , \mu_{2} (t) )    |\omega_{0}|(\R^2) .
\end{eqnarray*}
On the other hand  using Lemma \ref{LipLip} we obtain 
$$ \lambda_b (t)  \leqslant \|  \widetilde{H}  \|_{\Lip} (1+ |\omega_{0}|(\R^2)  ) \lambda (t) .$$
From the equations \eqref{Cara0}-\eqref{Cara2} we infer that $\lambda (0) = 0$. Then the Gronwall lemma leads to
\begin{eqnarray*}
\lambda (t) &\leqslant& C \int_0^t e^{C(t-s)} {W}_1 ( \mu_{1}  (s) , \mu_{2} (s) ) ds ,
\end{eqnarray*}
for any  $t$ in $[0,T]$, where $C >0$ depends only on $ \|  \widetilde{H}  \|_\Lip$ and $|\omega_0|(\R^2)$. Since, for any couple $(\nu_1,\nu_2)\in\mathcal{M}_1(\R^2)\times\mathcal{M}_1(\R^2)$, $|U[\nu_1]-U[\nu_2]|= |\widetilde{K}[\nu_1]|$, the same computation gives a similar bound for $ \widetilde{\lambda} (t)$.

Now, for any  $t$ in $[0,T]$, denoting by the index $\varphi$ a supremum over the unit ball of $\text{Lip}(\R^{2})$ and the index $\psi$ a supremum over the unit ball of $\text{Lip}(\R^{2}\times\R^2)$.
\begin{eqnarray*}
 {W}_1 ( \mathcal{T}_{\mu_0}  \lbrack \mu_{1}  \rbrack (t) ,  \mathcal{T}_{\mu_0}  \lbrack \mu_{2}  \rbrack (t)  )  
&=& W_1({\phi_t^{\mu_1}}_{\hspace{-1.8mm}\#}\omega_0,{\phi_t^{\mu_2}}_{\hspace{-1.8mm}\#}\omega_0)+W_1({\Sigma_t^{\mu_1}}_{\hspace{-1.8mm}\#} f_0,{\Sigma_t^{\mu_2}}_{\hspace{-1.8mm}\#} f_0)\\
 &=&  \sup_{\varphi }  \,   \int_{\R^2  } \varphi\,  d(  {\phi_t^{\mu_1}}_{\hspace{-1.8mm}\#}\omega_{0}-  {\phi_t^{\mu_2}}_{\hspace{-1.8mm}\#}\omega_{0} )   
 +  \sup_{ \psi }  \,   \int_{\R^4  } \psi\,  d(  {\Sigma_t^{\mu_1}}_{\hspace{-1.8mm}\#}f_{0}-  {\Sigma_t^{\mu_2}}_{\hspace{-1.8mm}\#}f_{0} )    ,
 \\  &=&  \sup_{ \varphi  }  \,   \int_{\R^2  } ( \varphi  \circ  \phi_t^{\mu_1}  -  \varphi  \circ \phi_t^{\mu_2}) d \omega_{0}+  \sup_{\psi }  \,   \int_{\R^4  } (\psi   \circ \Sigma_t^{\mu_1}-\psi   \circ   \Sigma_t^{\mu_2} ) df_0,
 \\ &\leqslant  &  \lambda (t)  + \widetilde{\lambda} (t) 
  \\  &\leqslant& C \int_0^t e^{C(t-s)} {W}_1 ( \mu_{1}  (s) , \mu_{2} (s) )   ds,
\end{eqnarray*}

Therefore
\begin{eqnarray*}
 \mathcal{W}_1 ( \mathcal{T}  \lbrack \mu_{1}  \rbrack ,  \mathcal{T}  \lbrack \mu_{2}  \rbrack  )  
 \leqslant (e^{CT} -1) \mathcal{W}_1 ( \mu_{1}  ,  \mu_{2}  ) ,
 \end{eqnarray*}
so that $\mathcal{T} $ is a contraction for $T$ small enough.
This smallness condition only depends on the total variation of $\omega_{0}$ and on the Lipschitz modulus of $\widetilde{H}$. Since $\omega(t)$ remains compatible with $\omega_0$ at all time, it has in particular always the same total variation. We hence infer the global existence part of the result by an iteration process.

Let us now prove the stability estimate \eqref{stab}.
So let us consider two solutions  $  \mu_1 := ( \omega_1 , f_1) $ and $\mu_2 :=  ( \omega_2 , f_2) $  of the systems   \eqref{VE1}-\eqref{VE2}-\eqref{VE3R} associated to two initial data  $ \mu_{0}^1  := ( \omega_{0}^1 , f_{0}^1) $ and $ \mu_{0}^2 := ( \omega_{0}^2 , f_{0}^2) $. Then, we have the following where  $t$ in $[0,T]$ is understood, 
\begin{eqnarray}
\label{decoup}
W_{1} (  \mu_1 (t) ,\mu_2 (t) ) & =&  {W}_1 ( {\phi_t^{\mu_1}}_{\hspace{-1.8mm}\#}\omega_0^1,{\phi_t^{\mu_2}}_{\hspace{-1.8mm}\#}\omega_0^2) + W_1({\Sigma_t^{\mu_1}}_{\hspace{-1.8mm}\#} f_0^1,{\Sigma_t^{\mu_2}}_{\hspace{-1.8mm}\#} f_0^2).
\end{eqnarray}
Let us first focus our attention on the first term in the right hand side. We have by the triangle inequality that 
\begin{eqnarray*}
{W}_1 ( {\phi_t^{\mu_1}}_{\hspace{-1.8mm}\#}\omega_0^1,{\phi_t^{\mu_2}}_{\hspace{-1.8mm}\#}\omega_0^2) \leqslant {W}_1 ( {\phi_t^{\mu_1}}_{\hspace{-1.8mm}\#}\omega_0^1,{\phi_t^{\mu_2}}_{\hspace{-1.8mm}\#}\omega_0^1) + {W}_1 ( {\phi_t^{\mu_2}}_{\hspace{-1.8mm}\#}\omega_0^1,{\phi_t^{\mu_2}}_{\hspace{-1.8mm}\#}\omega_0^2).
\end{eqnarray*}
Now the first term can be tackled as previously whereas the second one can be bounded by using  Lemma \ref{CompoLip} and the estimate \eqref{Lip0}.
We proceed in the some way for the second term  in the right hand side of \eqref{decoup}.
This proves the desired estimate.

The last statement of the Theorem follows classically from basic transport theory. 
\end{proof}
\subsection{Mean-field limit for regularized kernels }
We infer from Theorem \ref{MFR} the following result about the mean-field limit in the regularized case where \eqref{VE3R} is taken instead of  \eqref{VE3}.
\begin{Corollary}
Assume that $( \omega_{0} , f_{0}) $ is in $\mathcal{M}_1 ( \R^2 ) \times \mathcal{P}_1 ( \R^2 \times  \R^2 ) $.  
Let $( h_{i}^0 ,  h_{i}^1 )_{i\in\N^*} \in (\R^2 \times  \R^2 )^{\N^{*}}$ such that the sequence of empirical measure $(f_{0}^N)_{N \in \N^*}$, defined by
\begin{align*}
f_{0}^N := \frac{1}{N} \sum_{i=1}^{N } \delta_{(h_{i}^0  , h_{i}^1 )} \in \mathcal{P}_1(\R^2\times\R^2),
 \end{align*} 
satisfies
$W_{1} (f_{0}^N , f_{0} ) \rightarrow 0 \text{ when } N  \rightarrow  +\infty .$
Let us denote by $\mu_{N} := (\omega^{N} , f^{N}  )_{N \in\N^*} $, $\mu := (\omega , f)$ the solutions respectively associated to the initial data $( \omega_{0} , f_0^N  )_{N \in\N^*}$, $( \omega_{0}  ,f_{0})$ given by Theorem \ref{MFR}.
Then 
\begin{enumerate}
\item for any $t>0$,  for any $N \geq 1$, 
\begin{eqnarray*}
f^{N} (t) =  \frac{1}{N} \sum_{i=1}^{N } \delta_{(h_{i,N} (t) , h'_{i,N}(t) )} ,
\end{eqnarray*}
where 
\begin{eqnarray}
 \label{N2RT} h_{i,N} '' (t) &=&   \Big( h_{i,N} '(t) -  \widetilde{u}^{N} (t, h_{i,N} (t)  )  \Big)^\perp   , 
 \\  \label{N1RuT} \widetilde{u}^{N} &=& \widetilde{K} \lbrack \omega^{N} \rbrack + \frac{1}{N}  \sum_{j=1}^N \widetilde{H}  ( \cdot - h_{j,N} (t) ), 
\\  	\label{N4RT}( h_{i,N}  (0) ,h'_{i,N}  (0) )   &=& (h_{i}^0 ,h_{i}^1) .
\end{eqnarray}
\item for any $T>0$, 
$\mathcal{W}_1 (\mu_{N} , \mu  ) \rightarrow 0$  when $N  \rightarrow  +\infty $, where $\mathcal{W}_1$ is the distance defined in \eqref{DisT}.
\end{enumerate}
\end{Corollary}
Let us recall here that $ \widetilde{H}(0)=0$ so that the velocity $\widetilde{u}_{N}$ can be seen as a regularization of the velocity $u$ of Section  \eqref{M-fl} and of the velocities $v_i$ as well. Hence the equations  \eqref{N2RT}-\eqref{N4RT}, together with the following transport equation for the vorticity:
\begin{eqnarray*}
\partial_t \omega_N + \div (\omega_N \,  \widetilde{u}_N) = 0 ,
\end{eqnarray*}
can be seen as a regularization of the equations of Section  \eqref{M-fl}.\vspace{2mm}\\

It is also possible to obtain the following result for the case where the fluid vorticity is also discretized.
\begin{Corollary}
Assume that $( \omega_{0} , f_{0}) $ is in $\mathcal{M}_1 ( \R^2) \times \mathcal{P}_1 ( \R^2  \times  \R^2 ) $.  
Let $(x_{i}^0 , \alpha_i , h_{i}^0 ,  h_{i}^1 )_{i \in\N^*}\in (\R^2 \times \R \times \R^2 \times  \R^2 )^{\N^{*}}$ and consider the sequence $(\omega_0^N ,f_0^N )_{N\in\N^*}\in \mathcal{M}_1 ( \R^2 ) \times  \mathcal{P}_1 ( \R^2  \times  \R^2 ) $ defined by
\begin{eqnarray*}
\omega_{N}^0 &:=&  \frac{\omega_{0}^+(\R^2)}{\beta_N^+} \sum_{i \in I_N^+ } \alpha_i  \delta_{x_{i}^0 } 
+ \frac{\omega_{0}^-(\R^2)}{\beta_N^-} \sum_{i \in I_N^- } \alpha_i  \delta_{x_{i}^0 },
 \\
 f_{0}^N &:=& \frac{1}{N} \sum_{i=1}^{N } \delta_{(h_{i}^0  , h_{i}^1 )},
\end{eqnarray*}
where
 $I_N^\pm :=   \{ i \in  \llbracket  1 , N\rrbracket \,:\,  \pm  \alpha_i   > 0 \}   \text{ and } 
\beta_N^\pm :=  \sum_{i \in  I_N^\pm} \alpha_i    $. Assume that $W_{1} ( \omega_{0}^N , \omega_{0} )  + W_{1} (f_{0}^N , f_{0} ) \rightarrow 0$  when  $N  \rightarrow  +\infty $.
Let us denote by $\mu_{N} := (\omega^{N} , f^{N}  )_{N \in\N^*} $, $\mu := (\omega , f)$ the solutions respectively associated to the initial data $( \omega_{0}^N , f_{0}^N  )_{N \in\N^*}$, $( \omega_{0}  ,f_{0})$ given by Theorem \ref{MFR}.
Then 
\begin{enumerate}
\item for any $t>0$,  for any $N \geq 1$, 
\begin{eqnarray*}
\omega^{N} (t) =   \frac{\omega_{0}^+(\R^2)}{\beta_N^+} \sum_{i \in I_N^+ } \alpha_i  \delta_{x_{i,N} (t)} 
+ \frac{\omega_{0}^-(\R^2)}{\beta_N^-} \sum_{i \in I_N^- } \alpha_i  \delta_{x_{i,N} (t)}   \text{ and } f_{N} (t) =  \frac{1}{N} \sum_{i=1}^{N } \delta_{(h_{i,N} (t) , h'_{i,N} (t) )} 
\end{eqnarray*}
where 
\begin{eqnarray}
 \label{N1RuTv} x'_{i,N} (t) &=&   \widetilde{u}^{N} (t,x_{i,N} (t)  ),
\\ \label{N2RTv} h_{i,N} '' (t) &=&   \Big( h_{i,N}'(t) -    \widetilde{u}^{N} (t, h_{i,N} (t)  )  \Big)^\perp   , 
 \\  \label{N3RuTv}   \widetilde{u}^{N}&=& \widetilde{K} \lbrack \omega^{N} \rbrack + \frac{1}{N}  \sum_{1 \leq j \leq N } \widetilde{H}  ( \cdot - h_{j,N} (t) ), 
\\  	\label{N4RTv}( x_{i,N} (0) , h_{i,N}  (0) ,h'_{i,N}  (0) )   &=& ( x_{i}^0 ,h_{i}^0 ,h_{i}^1) .
\end{eqnarray}
\item for any $T>0$, 
$\mathcal{W}_1 (\mu_{N} , \mu  ) \rightarrow 0$  when $N  \rightarrow  +\infty $, where $\mathcal{W}_1$ is the distance defined in \eqref{DisT}.
\end{enumerate}
\end{Corollary}
The system  \eqref{N1RuTv}-\eqref{N4RTv} describe the dynamics of $2N$  pointwise vortices localized in $(x_{i,N} (t) )_{1 \leq i \leq N} $ and in $(h_{i,N} (t) )_{1 \leq i \leq N} $ under pairwise regular interaction, the first $N$ ones being massless. 

\subsection{Stability of the hydrodynamic regime}
\label{Hydro}
Given a map $v:\R^d\rightarrow\R^d$, one can define its ``graph'' function 
\begin{align*}
\tau^v:\R^d&\longrightarrow\R^d\times\R^d\\
x &\longmapsto (x,v(x)).
\end{align*}
For any measure $\rho\in\mathcal{M}(\R^d)$, the pushforward measure $\tau^v_\# \rho$ is hence a well-defined measure on $\R^d\times\R^d$. In the particular case where $v$ is a continuous function, one may write, for any bounded continuous test function $\phi:\R^d\times\R^d\rightarrow\R$,
\begin{align*}
\langle \tau^v_\# \rho,\phi\rangle =\langle \rho,\phi\circ\tau^v\rangle  = \int_{\R^d} \phi(x,v(x)) d\rho(x),
\end{align*}
so that the measure $\tau^v_\# \rho$ can hence be seen heuristically as the product $\rho(x) \delta_{v(x)}(\xi)$ and may be called \emph{monokinetic}. The generalization to the case of time-dependent measures/functions is straightforward. Actually, another corollary of Theorem  \ref{MFR} deals with the case of solutions taking the previous form. Formally, in a time-dependent frame, if $f:=\tau^v_\#\rho$ the equations \eqref{VE1}-\eqref{VE3} reduce to the following system for the unknowns $\omega(t,x)$, $\rho(t,x)$ and $v(t,x)$: 
\begin{eqnarray*}
 \partial_t \omega + \div_x (\omega  u)  &=& 0 ,
\\   \partial_t  \rho  + \div_x ( \rho v ) &=& 0 ,
\\   \partial_t  (\rho v)  + \div_x ( \rho v \otimes v) &=&\rho (v-u)^\perp  ,
\end{eqnarray*}
where $u$ is still given by  
\begin{eqnarray}
\label{noyau}
u = K  \lbrack \omega + \rho \rbrack .
\end{eqnarray}
When $\rho$ does not vanish the third equation in the system above can be simplified into
\begin{eqnarray*}
   \partial_t  v  +  v \cdot \nabla_x v &=& (v-u)^\perp  .
\end{eqnarray*}
Here again we will deal with the case of a regularized kernel substituting the law
\begin{eqnarray}
\label{noyauR}
u = \widetilde{K}  \lbrack \omega + \rho \rbrack 
\end{eqnarray}
to \eqref{noyau}.
For such a system we have the following result of local-in-time result.
\begin{Proposition}[] 
\label{HydroRClassiq}
Let  $( \omega_{0} , \rho_{0} , v_{0} ) \in \mathcal{M}_1 ( \R^2) \times  \mathcal{P}_1 ( \R^2)  \times  \Lip  ( \R^2)    $.
Then there exists $T >0$ and $ (\omega , \rho , v ) \in L^{\infty} (0, T ; \mathcal{M}_1 ( \R^2) \times   \mathcal{P}_1 ( \R^2)  \times   \Lip  ( \R^2)   )$ solution of
\begin{eqnarray*}
\partial_t \omega + \div_x (\omega  u)  &=& 0 ,
\\   \partial_t  \rho  + \div_x ( \rho v ) &=& 0 ,
\\   \partial_t  v  +  v \cdot \nabla_x v &=& (v-u)^\perp  ,
\end{eqnarray*}
with  $ (\omega (0) ,\rho (0) , v (0) ) = (\omega_{0} , \rho_{0} , v_{0} )$ and  $u $ given by \eqref{noyauR}.
\end{Proposition}
\begin{proof}
The proof is an easy application of the  
method of characteristics.
\end{proof}
Then we have the following.
\begin{Corollary}[]
\label{Hydr}
Assume that $( \omega_{0} , \rho_{0} , v_{0} ) \in \mathcal{M}_1 ( \R^2) \times  \mathcal{P}_1 ( \R^2) \times  \Lip  ( \R^2)    $.
Let $T>0$ and $ (\omega , \rho , v ) \in L^{\infty} (0, T ; \mathcal{M}_1 ( \R^2) \times   \mathcal{P}_1 ( \R^2)  \times   \Lip  ( \R^2)   )$ be given by
Theorem \ref{HydroRClassiq}. Consider  $f_{0}:={\tau^{v_0}}_{\hspace{-1mm}\#}\rho_0 \in  \mathcal{P}_1 ( \R^2  \times  \R^2 )$. Let be given a sequence  $(f_{0}^N)_{N \geqslant 1} $ is in $ \mathcal{P}_1 ( \R^2  \times  \R^2 ) $ such that $W_{1} (f_{0}^N , f_{0} ) \rightarrow 0 \text{ when } N  \rightarrow  +\infty$. Let us denote by $\mu^{N} := (\omega^{N} , f^{N}  )_{N \geq 1} $, $\mu := (\omega , f)$ the solutions respectively associated to the initial data $( \omega_{0} , f_{0}^N  )_{N \geq 1}$, $( \omega_{0}  ,f_{0})$ given by Theorem \ref{MFR}.
Then, $f$ actually equals to $\tau^v_\#\rho$ and, for the time $T$ given above and the distance $\mathcal{W}_1$ defined by \eqref{DisT}, $\mathcal{W}_1 (\mu^{N} , \mu  ) \rightarrow 0$  when $N  \rightarrow  +\infty $.
\end{Corollary}
\section{Existence results}
\label{existence}

\subsection{Weak solutions}
\label{PasYounik}
The main result of this section is the following 
\begin{Theorem}
\label{PasLoep}
If  $\omega_0 $ is in $(L^\frac{4}{3} \cap L^1 ) ( \R^{2})$ and $f_0$ is in  $(L^\infty \cap L^1 ) ( \R^{2} \times  \R^{2})$ such that the kinetic energy of the dispersed phase is  finite:
\begin{align*}
 \int_{\R^2 \times \R^2 } f_0 (x,\xi) |\xi|^2 dx d \xi < + \infty ,
\end{align*}
  then for any  $T>0$ there exists at least one
weak solution 
 $$ ( \omega , f) \in  \bigcap_{p\in[1,\infty[}\mathscr{C}^0\Big(  \lbrack 0,T]\,; L^{4/3}  ( \R^{2}  ) \times L^p ( \R^{2} \times  \R^{2}) \,\Big) $$
  to the equations \eqref{VE1}-\eqref{VE3}. Morover  the kinetic energy of the dispersed phase  of this solution  is finite at any time: 
\begin{align*}
\forall t\in[0,T],\quad \int_{\R^2 \times \R^2 } f (t,x,\xi) |\xi|^2 dx d \xi < + \infty .
\end{align*}
\end{Theorem}
\begin{proof}
Let us introduce some notations. Given two functions $g(\xi) $ and $ h (x,\xi)$, and a nonnegative number $\alpha$, we denote:
\begin{equation*}
m_\alpha (g) :=  \int_{\R^{2}} g (\xi) |\xi|^\alpha d\xi  \text{ and } M_\alpha (h) :=  \int_{\R^{4}} h (x,\xi) |\xi|^\alpha dx d\xi  .
\end{equation*}
We have the following useful result
\begin{Lemma}\label{lem:interp} If $g \in
L^\infty(\R_+\times\R^2\times\R^2)$ (positive) such as
$m_\gamma(g)(t,x)$ is finite almost everywhere in time-space, then for
all $\alpha<\gamma$ we have, almost everywhere in time-space
\begin{align*}
m_\alpha(g)(t,x) \leq C(\|g\|_{L^\infty}+1)m_\gamma(g)(t,x)^{\frac{\alpha+2}{\gamma+2}}.
\end{align*}
\end{Lemma}
\begin{proof}
For all $R>0$
\begin{align*}
m_\alpha(g)(t,x) &= \int_{\xi \leq R} g(t,x,\xi) |\xi|^\alpha d \xi + \int_{\xi > R} g(t,x,\xi)|\xi|^\alpha d \xi\\
&\leq C \|g\|_{L^\infty}R^{\alpha+2}+\frac{1}{R^{\gamma-\alpha}}m_\gamma(g)(t,x),
\end{align*}
and the real $\displaystyle R=m_\gamma(t,x)^{\frac{1}{\gamma+2}}$ does the trick.
\end{proof}
Now consider $(\omega_0^n,f_0^n)$ a smooth compactly supported approximation of $(\omega_0,f_0)$. Recall the definition
\begin{align*}
H(x) := \frac{1}{2\pi}\frac{x^\perp}{|x|^2}.
\end{align*}
Consider a regularizing kernel $(\varphi_n)_{n\in\N}$ of pair functions, in such a way that for all $n\in\N$ 
\begin{align*}
H_n := H\star\varphi_n \in W^{1,\infty}(\R^2),\quad\text{and}\quad H_n(0)=0.
\end{align*}

Theorem \ref{MFR} allows us to construct for every $n\in\N$ a global strong solution $(\omega^n,f^n)$ to the system
\begin{align}
\label{eq:approx1}& \partial_t \omega^n + u^n \cdot \nabla_x \omega^n = 0,\quad \forall (t,x)\in[0,T]\times\R^2,\\
\label{eq:approx2} & \omega^n(0,x) = \omega_{0}^n(x),\quad \forall x\in\R^2,\\
\label{eq:approx3} & \partial_t f^n +\xi \cdot \nabla_x f^n +(\xi-u^n)^\perp \cdot \nabla_\xi f^n = 0,\quad \forall (t,x,\xi)\in[0,T]\times\R^2 \times\R^2, \\
\label{eq:approx4}&f^n(0,x,\xi) = f_{0}^n(x,\xi),\quad \forall (x,\xi)\in\R^2\times\R^2,
\end{align}
where $u^n:=K_n[\rho^n+\omega^n]$ and $K_n$ is defined by $K_n[g] = K[g]\star\varphi_n = H_n\star g $, with $(\varphi_n)_{n\in\N}$.  Note that $\omega^n$ and $f^n$ are compactly supported (and the later positive), and we have, for all $n\in\N$,
\begin{align}
\label{ineq:flinfty}\|f^n\|_{L^\infty} &= \|f_0^n\|_{L^\infty}\leq \|f_0\|_{L^\infty}, \\
\label{ineq:omegal43}\|\omega^n\|_{L^\infty_{t}(L^{4/3}_x)} &= \|\omega_0^n\|_{L^{4/3}}\leq \|\omega_{0}\|_{L^{4/3}},\\
\label{ineq:rhol1}\|\rho^n\|_{L^\infty_t(L^1_x)} &= \|f_{0}^n\|_{L^1} \leq \|f_{0}\|_{L^1}.
\end{align}
From now on, $C_{\text{in}}^T$ will denote a positive constant (that may vary from line to line) depending only on the initial data and $T$.\vspace{2mm}\\
Denoting $\displaystyle \rho^n:=\int_{\R^2} f^n d\xi$, $\displaystyle j^n:=\int_{\R^2} f^n \xi d\xi$, and applying lemma \ref{lem:interp} with respectively $\alpha=0=2-\gamma$ and $\alpha=1=\gamma-1$ we get
\begin{align}
\label{ineq:interp1}\|\rho^n(t)\|_{L^{2}} &= \|m_0(f^n)(t)\|_{L^{2}}\leq  C_{\text{in}} M_2(f^n)(t)^{1/2},\\
\label{ineq:interp2}\|j^n(t)\|_{L^{4/3}} &\leq \|m_1(f^n)(t)\|_{L^{4/3}}\leq  C_{\text{in}} M_2(f^n)(t)^{3/4}.
\end{align}
Using H\"older inequality and \eqref{ineq:rhol1}--\eqref{ineq:interp1} we get
\begin{align*}
\|\rho^n(t)\|_{L^{4/3}} \leq \sqrt{\|\rho^n(t)\|_{L^1}} \sqrt{\|\rho^n(t)\|_{L^2}} \leq \CinT M_2(f^n)(t)^{1/4}.
\end{align*}
Thanks to the Hardy-Littlewood-Sobolev inequality we get 
\begin{align}
\label{ineq:hardysob}\|u^n(t)\|_{L^4} \leq C \|\rho^n(t) + \omega^n(t)\|_{L^{4/3}} \leq C_{\text{in}}^T[1+M_2(f^n)(t)^{1/4}].
\end{align}
Now, multiplying the kinetic equation by $|\xi|^2$ and integrating in space-velocity gives
\begin{align*}
\frac{d}{dt}M_2(f^n)(t) = -2 \int_{\R^2}j^n(t,x) \cdot {u^n}^\perp (t,x)dx\leq 2 \|j^n(t)\|_{L^{4/3}} \|u^n(t)\|_{L^4},
\end{align*}
and hence, using \eqref{ineq:interp2}, 
\begin{align*}
4\frac{d}{dt}\Big[M_2(f^n)(t)^{1/4}\Big]=M_2(f^n)(t)^{-3/4}\frac{d}{dt} M_2(f^n)(t) &\leq \CinT \|u^n(t)\|_{L^4 _x},
\end{align*}
so we get 
\begin{align*}
 M_2(f^n)(t) &\leq \Big[M_2(\fin)^{1/4}+ \fin\int_0^t \|u^n(s)\|_{L^4} ds\Big]^4\\
&\leq \CinT\Big[1+ \int_0^t \|u^n(s)\|_{L^4}^4 ds\Big],
\end{align*}
and using \eqref{ineq:hardysob} we have finally 
\begin{align*}
 M_2(f^n)(t) &\leq \CinT\Big[1+ \int_0^t  M_2(f^n)(s) ds\Big],
\end{align*}
which, \emph{via} a linear Gr\"onwall lemma, allows us to conclude that $(M_2(f^n))_{n\in\N}$ is bounded in $L^\infty([0,T])$. Using \eqref{ineq:interp1}, \eqref{ineq:interp2} and \eqref{ineq:hardysob} we hence deduce the boundedness of $(\rho^n)_{n\in\N}$, $(j^n)_{n\in\N}$ and $(u^n)_{n\in\N}$ respectively in $L^\infty_t(L^2_x)$, $L^\infty_t(L^{4/3}_x)$ and $L^\infty_t(L^4_x)$.

Now \eqref{eq:approx1} gives $\partial_t \omega^n = - \text{div}_x(u^n \omega^n)$, and because $(\omega^n)_{n\in\N}$ is bounded in $L^\infty_t(L^{4/3}_x)$, we have that $(\partial_t \omega^n)_{n\in\N}$ is bounded in $L^\infty_t(\mathscr{C}^1_c(\R^2)')\hookrightarrow L^\infty_t(H^{-m}_{x,\text{loc}})$ for $m$ large enough. Analogously, we have $\partial_t \rho^n = - \text{div}_x j^n$ and hence $(\partial_t \rho^n)_{n\in\N}$ is bounded in the same space $L^\infty_t(H^{-m}_{x,\text{loc}})$. Now $\cdot \star \nabla^\perp G$ sends $H^{-m}_{\text{loc}}(\R^2)$ in another $H^{-q}_{\text{loc}}(\R^2)$, so that we have finally $(\partial_t u^n)_{n\in\N}$ is bounded in $L^\infty_t(H^{-q}_{x,\text{loc}})$. But, from Calderon-Zygmund theory we know that $\| \nabla u^n(t)\|_{L^{4/3}_x} \leq C \|\rho^n(t)+\omega^n(t)\|_{L^{4/3}} \leq \CinT$,  and since $(u^n)_{n\in\N}$ is bounded in $L^\infty_t(L^4_x)$, it is also bounded in $L^\infty_t(W^{1,4/3}_{x,\text{loc}})$. We hence apply Aubin's lemma to get (up to an extraction) the strong convergence of $(u^n)_{n\in\N}$ in $L^\infty_t(L^{p}_{x,\text{loc}})$ for all $p<4$. At this stage, we cannot use a weak/strong type argument to pass to the limit in both equation: the bounds \eqref{ineq:flinfty} and \eqref{ineq:omegal43} insures that up to an extraction $(f^n)_{n\in\N}$ and $(\omega^n)_{n\in\N}$ are weakly converging respectively to some $f$ and $\omega$ in $L^\infty_{t}(L^\infty_{x,\xi})-\star$ and $L^\infty_t(L^{4/3}_x)-\star$, but $(u^n)_{n\in\N}$ does not \emph{a priori} converge in $L^1_t(L_x^4)$ so that we cannot directly pass to the limit in the nonlinear term of \eqref{eq:approx1}. However it follows from the  stability part of \cite{Dipernalions} that  $(\omega^n)_{n\in\N}$ and $(f^n)_{n\in\N}$ respectively converge in $\mathscr{C}^0_t(L_x^{4/3})$ and $\mathscr{C}^0_t(L_{x,\xi}^p)$, for all $p\in[1,\infty[$. We hence pass to the limit in both equations \eqref{eq:approx1}--\eqref{eq:approx2} and obtain by the same time a weak formulation including this time the initial data. This concludes the proof. %
\end{proof}

\subsection{Classical solutions}
In the spirit of \cite{degond} we construct regular solutions to our system, taking some regular initial data, with a kinetic phase having sufficient decreasing at infinity in the velocity variable. More precisely we have the following result:
\begin{Theorem}\label{thm:deg:strong}
Assume $\omega_0\in W^{1,1}(\R^2)\cap W^{1,\infty}(\R^2)$ and $f_0\in W^{1,1}(\R^2\times\R^2)$ (the latter being positive), such as 
\begin{align}\label{strongdecay}
(1+|\xi|^2)^{\gamma/2}(|f_0| + |Df_0|) \in L^\infty(\R^2\times\R^2),
\end{align}
for some real number $\gamma>2$. Let $T>0$. Then there exists a classical solution 
\begin{align*}
(\omega,f) \in L^\infty([0,T];W^{1,1}(\R^2)) \times L^\infty([0,T]; W^{1,1}(\R^2\times\R^2))
\end{align*}
to the equations \eqref{VE1}-\eqref{VE3} satisfying furthermore
\begin{align*}
(1+|\xi|^2)^{\gamma/2}(|f_0| + |Df_0|) \in L^\infty([0,T]; L^\infty(\R^2\times\R^2)),
\end{align*}
\end{Theorem} 
\begin{Remark}
\label{rem:deg} Taking $W^{1,m}$ initial data and replacing \eqref{strongdecay} by a similar assumption taking into account all the $m$-th first derivatives lead to even more regular solution. This is done in \cite{degond}, the proof being quite close to the case $m=1$.  
\end{Remark}
As we follow closely the method of \cite{degond} we will only provide a sketch of proof. More precisely
we will only prove some appropriate a priori bounds for the  approximations given, as in the previous section, by \eqref{eq:approx1}--\eqref{eq:approx4}.
\begin{proof}
We consider again the solutions of \eqref{eq:approx1}--\eqref{eq:approx4} and we define as in \cite{degond}
\begin{align*}
Y^n(t,x,\xi) := (1+|\xi|^2)^{\gamma/2}f^n(t,x,\xi), 
\end{align*}
we have, using \eqref{eq:approx3}
\begin{align*}
\partial_t Y^n + \xi\cdot \nabla_x Y^n + (\xi-u^n)^\perp \cdot \nabla_\xi Y^n = \gamma  {u^n}^\perp \cdot \xi (1+|\xi|^2)^{(\gamma-2)/2}f^n,
\end{align*}
which is just a classical transport (Vlasov) equation with a source term. We hence have, thanks to the assumption on the initial data
\begin{align*}
\|Y^n(t)\|_{L^\infty} \leq \|f_0 (1+|\xi|^2)^{\gamma/2}\|_{L^\infty} + \gamma \int_0^t \| (1+|\xi|^2)^{(\gamma-1)/2}f^n(s)\|_{L^\infty} \|u^n(s)\|_{L^\infty} d s.
\end{align*}
But $u^n:=K_n[\rho^n+\omega^n]$ so that 
\begin{align*}
\|u^n(s)\|_{L^\infty} &\leq C \|\rho^n(s)+\omega(s)\|_{L^1}^{1/2}\|\rho^n(s)+\omega^{n}(s)\|_{L^\infty}^{1/2}\\
&\leq \big[\|\rho_0\|_{L^1}+ \|\omega_0\|_{L^1}\big]^{1/2} \big[\|\rho^n(s)\|_{L^\infty}+\|\omega_0\|_{L^\infty}\big]^{1/2},
\end{align*}
where we used for the first inequality a classical interpolation estimate true for the Biot-Savart operator and hence for its approximation $K_n$, uniformly in $n$. The second inequality is a direct consequence of equations \eqref{eq:approx1} and \eqref{eq:approx3} and the definition of $f_0^n$ and $\omega_0^n$ as regularization of the initial data. We hence use, as in \cite{degond} (see appendix) the interpolation estimate
\begin{align*}
\|\rho^n(s)\|_{L^\infty} \leq C_\gamma \|f^n(s)\|_{L^\infty}^{(\gamma-2)/\gamma}\|Y^n(s)\|_{L^\infty}^{2/\gamma}.
\end{align*}
We have also
\begin{align*}
\| (1+|\xi|^2)^{(\gamma-1)/2}f^n(s)\|_{L^\infty} \leq C_\gamma \|f^n(s)\|_{L^\infty}^{1/\gamma} \|Y^n(s)\|_{L^\infty}^{(\gamma-1)/\gamma},
\end{align*}
so that we have finally a Gr\"onwall estimate
\begin{align*}
\|Y^n(t)\|_{L^\infty} \leq C\left[ 1  + \int_0^t \|Y^n(s)\|_{L^\infty} ds \right].
\end{align*}
We have hence at this stage $L^\infty_{t,x}$ bounds for $(Y^n)_{n}$, $(\rho^n)_{n}$ and $(u^n)_n$.
Now introduce, again as in \cite{degond} 
\begin{align*}
Z^n(t,x,\xi) := (1+|\xi|^2)^{\gamma/2}\nabla_{x,\xi} f^n(t,x,\xi).
\end{align*}
We have the vectorial equality
\begin{align}
\label{eq:vectZ}\partial_t  \nabla_{x,\xi}  f^n +\xi \cdot  \nabla_x  \nabla_{x,\xi}  f^n +(\xi-u^n)^\perp \cdot \nabla_\xi  \nabla_{x,\xi}  f^n = A^n  \nabla_{x,\xi} f^n,
\end{align}
where $A^n\in\mathscr{M}_6(\R)$ is the matrix
\begin{align*}
\begin{pmatrix}
0_3 & \nabla_x {u^n}^\perp \\
-\textnormal{Id}_3 &  -\nabla_{\xi} \xi^\perp 
\end{pmatrix},
\end{align*}
and hence, 
\begin{align*}
\partial_t Z^n + \xi\cdot \nabla_x Z^n + (\xi-u^n)^\perp \cdot \nabla_\xi Z^n = \gamma  {u^n}^\perp \cdot \xi  (1+|\xi|^2)^{(\gamma-2)/2} \nabla_{x,\xi} f^n + A^n Z^n.
\end{align*}
Using first a logarithmic estimate \`a la Beale-Kato-Majda, and then the previous bounds, we get:
\begin{align*}
\|\nabla_x u^n\|_{L^\infty} &\leq C\Big[1+ \|\rho^n(s)+\omega^n(s)\|_{L^\infty}(1+\ln\big\{1+\|\nabla_x\rho^n(s)+\nabla_x\omega^n(s)\|_{L^\infty}\big\}) +\|\rho^n(s)+ \omega^n(s)\|_{L^1}\Big]\\
&\leq C \big[1 + \ln\big\{1+\|\nabla_x\rho^n(s)\|_{L^\infty}+\|\nabla_x\omega^n(s)\|_{L^\infty}\big\}\big]
\end{align*}
In the same way
\begin{align*}
W^n(t,x) := \nabla_x \omega^{n}(t,x),
\end{align*}
verifies
\begin{align*}
\partial_t W^n + u^n \cdot \nabla_x W^n = - \nabla_x u^n W^n,
\end{align*}
so that we get eventually for some constant $C$, letting $\alpha_n(t):=\|Z^n(t)\|_{L^\infty} + \|W^n(t)\|_{L^\infty}$, 
\begin{align*}
\alpha_n(t) \leq C\Big[1+ \int_0^t \alpha_n(s) \big\{1+\ln \alpha_n(s)\big\}ds\Big], 
\end{align*}
which classically implies a bound for $(\alpha_n)_n$ in $L^\infty_t$, and hence a bound in $L^\infty_{t}(L^\infty_{x,\xi})$ for the sequence of matrices $(A^n)_n$. Integrating \eqref{eq:vectZ} and using the last bound we have also
\begin{align*}
\|\nabla_{x,\xi} f^n(t) \|_{L^1} \leq C\Big[1+ \int_0^t \|\nabla_{x,\xi} f^n(s) \|_{L^1} ds \Big],
\end{align*}
that is $(\nabla_{x,\xi}f^n)_n$ bounded in $L^\infty_t(L^1_{x,\xi})$. From all the previous bounds we get (up to several extractions) convergence in $L^\infty-w\star$ frame. 
From this the proof of Theorem \ref{thm:deg:strong} 
 follows exactly the same lines than Theorem $1.1$ in \cite{degond}.
\end{proof}
\section{A uniqueness result }
\label{Younik}
Let us give here a uniqueness result which extends the result \cite{Yudovich} by Yudovich about the incompressible Euler system and the result \cite{Loeper} by Loeper about the Vlasov-Poisson system.

\begin{Theorem}
\label{loep}
If  $\omega_0 $ is in  $(L^\infty \cap L^1)  ( \R^{2})$ and $f_0$ is in  $\mathcal{M}_{2}^+ ( \R^{2} \times  \R^{2})$  then, given $T>0$ there exists at most one
weak solution 
 $$( \omega , f) \in \mathscr{C}^0\Big(  [0,T];  (L^\infty \cap L^1) ( \R^{2}) -w\star \times \mathcal{M}^{+}  ( \R^{2} \times  \R^{2}) \Big) $$
  to the equations \eqref{VE1}-\eqref{VE3} such that  $ \rho$ is in  $L^\infty ( (0,T)  \times  \R^{2} )$. 
\end{Theorem}
Note that  Theorem \ref{loep} applies to the case of mono-kinetic solutions, that is for a density $f$ of the form given in Section \ref{Hydro}.
The proof of  Theorem \ref{loep} follows closely the method of \cite{Loeper}. Yet we provide a proof for sake of completeness.
\begin{proof}
Let be given two solutions $ ( \omega_1 , f_1) $ and $ ( \omega_2 , f_2) $ of  the equations   \eqref{VE1}-\eqref{VE3} such that  $ \rho_1 :=m_0(f_1) $ and $ \rho_2 :=  m_0(f_2) $  are  in $L^\infty ((0,T)  \times  \R^{2})$. 
Accordingly we define $u_{i} := K \lbrack \omega_{i} + \rho_{i} \rbrack $, for $i=1,2$.\vspace{2mm}\\
Recall the following classical result (see for example  \cite{CheminSMF}).
\begin{Proposition}\label{propdefKwhole}
Let  $g \in L^\infty (\R^2) \cap L^1 (\R^2)$.
Then $K [g] $ is  bounded and   Log-Lipschitz  on $\R^2$, that is  there exists $C>0$ such that $\|   K [g]  \|_{L^\infty}  \leq C $, and such that for any $(x,y) \in  \R^2 \times    \R^2  $ with $  | x-y | \leq \frac{1}{2}  $, there holds $ |   K [g]  (x) -  K [g]  (y)    |  \leq - C | x-y |  \ln  | x-y |  $.%
\end{Proposition}
It therefore follows from the assumptions of Theorem \ref{loep} that  the vector fields $u_{1}$ and   $u_{2}$ are bounded and Log-Lipschitz  on $\R^2$, uniformly in time, and so are the two vector fields, for $i=1,2$ 
\begin{align*}
U^{i} : [0,T] \times \R^2 \times \R^2  &\longrightarrow \R^{2}  \times \R^2 \\
(t,x,\xi) &\longmapsto  \Big(\xi , (\xi- u_{i} (t,x) )^{\perp} \Big).
\end{align*}
 Because of the mentioned regularity, referring again to \cite{CheminSMF}, we can define, for $i=1,2$, two unique continuous flow map $ \phi^{i} :[ 0,T]  \times \R^{2}$, and $\Sigma^i:[0,T]\times\R^2\times\R^2 \rightarrow\R^2\times\R^2$ such that
  \begin{align}
\label{flow1}
\phi^i_t(x):=\phi^{i} (t,x)  &=  x + \int^{t}_0 u_{i} (s, \phi^{i } (s,x) ) ds \\
\label{flow2}  \Sigma^{i}_t(x,\xi):=\Sigma^i(t,x,\xi) &= (x,\xi) + \int_0^t U^i(s,\Sigma^i(s,x,\xi)) ds.
  \end{align}
For $i=1,2$, define two vector fields $X^i,\Xi^i$, defined on $[0,T]\times\R^2\times\R^2$ with values in $\R^2$ such as $\Sigma^i=(X^i,\Xi^i)$ that is to say that $X^i$ and $\Xi^i$ are respectively the characteristic curves in the space and velocity variable. Denote as before $X^i_t(x,\xi):=X^i(t,x,\xi)$ and $\Xi^i_t(x,\xi):=\Xi^i(t,x,\xi)$ for all $(t,x,\xi)\in[0,T]\times\R^2\times\R^2$.\vspace{2mm}\\
We use the following result cf. Theorem $1.2$ in \cite{bc} or  \cite{AmBe} for a rather different proof.
\begin{Theorem}[] 
\label{Bahou}
Let $ \mathcal{V}(t,x) : [0,T] \times \R^d \rightarrow \R^d$ be a Borel vector field such that there exists $C \in L^1 (0,T)$ such that for all $t$ in $[0,T] $,
$\|    \mathcal{V}(t)   \|_{L^\infty}  \leq C(t)  $, and for any $(x,y) \in  \R^2 \times    \R^2$ with  $  | x-y | \leq \frac{1}{2} $,  $|  \mathcal{V}(t,x) -  \mathcal{V}(t,y)         |  \leq - C(t) | x-y |  \ln  | x-y |  $.
Let $\mu_0 \in \mathcal{M} ( \R^d) $.
Then there exists only one weakly continuous curve $\mu_t $ in $\mathscr{C}^0([0,T] ;  \mathcal{M} ( \R^d) \,  -w\star)$ solution of 
$\partial_t \mu_t + \div_x (\mu_t \mathcal{V} ) = 0 $.
This solution is given by 
$ \mu_t = {\Phi^\mathcal{V}_t}_{\hspace{-0.7mm}\#} \mu_0 $,
where $ \Phi_t^{\mathcal{V}}$ is given by flow associated to $\mathcal{V}$  at time $t$ by the formula 
  $\Phi_t^{\mathcal{V}}(x)  :=\Phi^{\mathcal{V}}(t,x)=  x + \displaystyle \int^{t}_0  \mathcal{V}(s,  \Phi^{\mathcal{V}}  (s,x) ) ds $.
\end{Theorem}
As a consequence, there holds, for $i=1,2$, $t\in[0,T]$,
\begin{eqnarray}
\label{pousse}
 ( \omega_i , f_i,\rho_i)  (t) &=& ( {\phi_{t}^i}_{\#} \omega_{0} , {\Sigma_t^i}_\# f_{0},{X_t^i}_{\#}  f_{0}), \\
\label{pousseprecis} ( \omega_i^{\pm} ,|\omega_i|)  (t) &=& ( {\phi_{t}^i}_{\#} \omega_{0}^{\pm} ,{\phi_{t}^i}_{\#} |\omega_{0}|),
\end{eqnarray}
so that at any time $t\in[0,T]$, $\omega_1(t)$, $f_1(t)$ and $\rho_1(t)$ are respectively compatible with $\omega_2(t)$, $f_2(t)$ and $\rho_2(t)$. For two time-dependent measure (always compatible) we will often use the shorthand $W_2(\nu_1,\nu_2)(t):=W_2(\nu_1(t),\nu_2(t))$.
Finally we introduce  the quantities
\begin{eqnarray*}
Q(t) :=   \int_{\R^2   \times  \R^{2} }  |  \Sigma_t^1  -  \Sigma_t^2   |^2 df_0 ,
\quad \widetilde{Q}(t) :=   \int_{\R^2}  |   \phi_t^1 - \phi_t^2 |^2   d|\omega_0| .
\end{eqnarray*}
Let us now recall the following result cf. \cite{Loeper}:
\begin{Lemma}
Let $\psi_i: \R^d \rightarrow  \R^d$, for $i=1,2$, be two homeomorphisms. 
Let $\nu  \in \mathcal{M}_+ ( \R^d)$ and denote, for $i=1,2$, $\nu_i := {\psi_i}_{\#} \nu$.
Then, for $\nu_{1}$ and $\nu_2$ are compatible and
\begin{eqnarray*}
W_2 (  \nu_{1} ,  \nu_{2} )^2  \leq \int_{\R^d}  |   \psi_1  - \psi_2 |^2   d\mu_0 .
\end{eqnarray*}
\end{Lemma}
Thanks to this lemma we obtain, for any $t\in[0,T]$,
\begin{eqnarray}
\label{inf1}
W_2 ( f_1,f_{2 })(t)^{2}  \leqslant  Q(t)  
 \text{ and }
 W_2 ( \omega_1  ,\omega_{2 } )(t)^{2} := W_2 ( \omega_1^{+} ,\omega_{2 }^{+})(t)^{2} + W_2 ( \omega_1^{-},\omega_{2 }^{-} )(t)^{2} \leqslant \widetilde{Q}(t),
\end{eqnarray}
so that it suffices to prove that $Q(t) $ and $ \widetilde{Q}(t)  $ vanish to get that $ ( \omega_1 , f_1) = ( \omega_2 , f_2) $.

Differentiating in time we get 
\begin{eqnarray*}
Q'(t)  &=& 2  \int_{\R^2   \times  \R^{2} } (\Sigma_t^1-\Sigma^2_t)\cdot (\partial_t \Sigma^1_t-\partial_t \Sigma^2_t) \,df_0 ,
\\  \widetilde{Q}'(t)    &=& 2  \int_{\R^2 }   (\phi^1_t-\phi^2_t)\cdot (\partial_t \phi^1_t-\partial_t \phi^2_t) \, d|\omega_0|.
\end{eqnarray*}
Using now \eqref{flow2} and \eqref{flow1} 
\begin{eqnarray*}
Q'(t)  &=& 2  \int_{\R^2   \times  \R^{2} }  (  X^1_t -  X^2_t   ) \cdot (  \Xi^1_t -  \Xi^2_t) \,d f_0
\\ &&- 2  \int_{\R^2   \times  \R^{2} }  (  \Xi^1_t  -  \Xi^2_t   ) \cdot \big(   u_{1}(t, X^1_t)  -  u_{2}(t,X^{2}_t)   \big)^{\perp}  df_0 ,
\\  \widetilde{Q}'(t)    &=& 2  \int_{\R^2 }   ( \phi^1_t - \phi^2_t ) \cdot  \big( u_{1} (t, \phi^{1 }_t)  - u_{2} (t, \phi^{2 }_t) \big) d|\omega_0|  .
\end{eqnarray*}
Using the Cauchy-Schwarz inequality, we have that the three terms  in the right hand sides above are respectively bounded by $Q(t)$, by 
\begin{eqnarray}
\label{preum1}
2 Q(t)^\frac{1}{2}  \Big( \int_{\R^2   \times  \R^{2} }  |     u_{1} (t,X^{1}_t  )  -  u_{2} (t,X^{2}_t )    | ^2df_0 \Big)^\frac{1}{2} ,
\end{eqnarray}
and by
\begin{eqnarray}
\label{preum2}
2 \widetilde{Q}(t)^\frac{1}{2} \Big( \int_{\R^2 }  |     u_{1} (t,\phi_{t}^1  )  -  u_{2} (t,\phi_{t}^2 )    | ^2d|\omega_0| \Big)^\frac{1}{2}.
\end{eqnarray}
Let us introduce the quantities
\begin{eqnarray*}
T_1 (t) &:=&  \int_{\R^2   \times  \R^{2} }  |  u_{2} (t,X^{1}_t )  -  u_{2} (t,X^{2}_t )    | ^2df_0 ,
\\ T_2 (t) &:=& \int_{\R^2   \times  \R^{2} } |  u_{2} (t,X^{1}_t )  -  u_{1} (t,X^{1}_t)    | ^2df_0 ,
\\  \widetilde{T}_1 (t) &:=&  \int_{\R^2  }   |    u_{2} (t, \phi^{1 }_t)  - u_{2} (t, \phi^{2 }_t)     | ^2 d|\omega_0|,
\\  \widetilde{T}_2 (t) &:=& \int_{\R^2  }    |u_{2} (t, \phi^{1 }_t)  - u_{1} (t, \phi^{1 }_t)     | ^2 d|\omega_0|.
\end{eqnarray*}
Then the terms in \eqref{preum1} and  \eqref{preum2} are respectively bounded by
\begin{eqnarray*}
2 Q(t)^\frac{1}{2} (T_1 (t)^\frac{1}{2} + T_2 (t)^\frac{1}{2} )
\text{ and }
2  \widetilde{Q}(t)^\frac{1}{2} ( \widetilde{T}_1 (t)^\frac{1}{2} +  \widetilde{T}_2 (t)^\frac{1}{2} ) .
\end{eqnarray*}
Using \eqref{pousse} and \eqref{pousseprecis} we have 
\begin{eqnarray*}
 T_2 (t) &=&  \int_{\R^2  }   |  u_{2} (t)  -  u_{1} (t)    | ^2 d\rho_1(t),
\\  \widetilde{T}_2 (t) &=& \int_{\R^2  }    |    u_{2} (t)  - u_{1} (t)     | ^2 d|\omega_1(t)| .
\end{eqnarray*}
The crucial estimate in Loeper's method is the following. 
\begin{Lemma}[\cite{Loeper}, Th. $4.5$] 
Consider two compatible signed measures on $\R^2$, $\mu_1$ and $\mu_2$, both having bounded densities with respect to he Lebesgue measure on $\R^2$.
Then
\begin{eqnarray*}
\| K \lbrack  \mu_1 \rbrack  -  K \lbrack  \mu_2 \rbrack   \|_{L^{2 }}
 \leqslant  (\max \{ \|  \mu_1^{+} \|_{L^{\infty }} , \|  \mu_2^{+} \|_{L^{\infty }} \})^{\frac{1}{2}} \ W_{2 } (\mu_1^{+} ,\mu_2^{+} )
  +   (\max \{ \|  \mu_1^{-} \|_{L^{\infty }} , \|  \mu_2^{-} \|_{L^{\infty }} \})^{\frac{1}{2}}   \  W_{2 } (\mu_1^{-} ,\mu_2^{-} )  .
\end{eqnarray*}
\end{Lemma}
In particular we get here, for every $t\in[0,T]$
\begin{eqnarray*}
\| u_1(t) - u_{2 }(t) \|_{L^{2 }} \leqslant  C \Big[W_2 ( \omega_1^{+} ,\omega_{2 }^{+}) +  W_2 ( \omega_1^{-} ,\omega_{2 }^{-})   +  W_{2 } (  \rho_{1},\rho_{2} )  \Big](t),
\end{eqnarray*}
where the constant $C$ depends on $\max (\|  \omega_1  \|_{L^{\infty }} ,\|    \rho_{1}   \|_{L^{\infty }} ,\|  \omega_2   \|_{L^{\infty }}, \| \rho_{2}   \|_{L^{\infty }} )$.
Modifying $C$ into another constant depending only on the same quantity, we get:
\begin{eqnarray*}
\max (T_2 (t) ,  \widetilde{T}_2 (t) )\leq C  \Big[W_2 ( \omega_1^{+} ,\omega_{2 }^{+} ) +  W_2 ( \omega_1^{-} ,\omega_{2 }^{-} )  +  W_{2 } (  \rho_{1} ,\rho_{2} )   \Big]^2(t) .
\end{eqnarray*}
Using now \eqref{inf1}  we obtain that for an appropriated constant $C$, 
\begin{eqnarray*}
\max (T_2 (t) ,  \widetilde{T}_2 (t) )\leq C ( Q(t) +   \widetilde{Q} (t)) .
\end{eqnarray*}
Let us now tackle the terms $T_1 $ and $ \widetilde{T}_1 $. 
From the $L^\infty$ bounds on $u_1,u_2$, and the equations \eqref{flow1} and \eqref{flow2} we infer that there exists $t_\star>0$ (depending only on the initial data through their norms in $L^1\cap L^\infty$) such that 
\begin{align*}
\|  \phi^{1} -  \phi^{2}  \|_{L^\infty ( [0,t_\star] \times \R^2)} + \|  X^{1} - X^{2}  \|_{L^\infty ( [0,t_\star] \times \R^2 \times \R^2)}   \leq  \frac{1}{e}.
\end{align*}
We  may even assume that the same estimate holds for $Q+\widetilde{Q}$ on the same interval.
Then, since the vector fields $u_{1}$ and   $u_{2}$ are   Log-Lipschitz  on $\R^2$, uniformly in time, we have, for all $t\in[0,t_\star]$, noting $g:x\mapsto x(\ln x)^2$
\begin{eqnarray*}
T_1 (t)  &\leqslant&   C \int_{\R^2   \times  \R^{2} } \left| (X^{1}_t -   X^2_t )  \ln ( X_{t}^1   -   X_{t}^2 )   \right|^2 df_0 = \frac{C}{4} \int_{\R^2   \times  \R^{2} } g\left(|X^{1}_t -   X^2_t|^2\right) df_0
\\  \widetilde{T}_1 (t)  &\leqslant&   C  \int_{\R^2  } | (\phi_{t}^1   - \phi_{t }^2  ) \ln  (\phi_{t }^1  - \phi_{t }^2  )  | ^2d|\omega_0| = \frac{C}{4} \int_{\R^2 } g\left(|\phi^{1}_t -   \phi^2_t|^2\right) d|\omega_0|
\end{eqnarray*}
Since $g$ is concave for $0 \leqslant x  \leqslant \frac{1}{e}$ it follows from  Jensen's inequality that (after renormalization) 
\begin{eqnarray*}
T_1 (t)  &\leqslant&  \frac{C}{4} \|f_0\|_{L^1_{x,\xi}} g\left\{ \int_{\R^2\times\R^2}|X^{1}_t -   X^2_t|^2 \frac{df_0}{\|f_0\|_{L^1_{x,\xi}}} \right\}.
\end{eqnarray*}
Since $g$ is increasing for $0 \leqslant x\leqslant \frac 1 e$, we get on $[0,t_\star]$
\begin{align*}
T_1 \leqslant \frac{C}{4} \|f_0\|_{L^1_{x,\xi}} g\left\{ \frac{Q}{\|f_0\|_{L^1_{x,\xi}}} \right\}=\frac{C}{4}Q(\ln Q-\ln\|f_0\|_{L^1_{x,\xi}})^2.
\end{align*}
Thus, doing the same for $\widetilde{Q}$, we get eventually
\begin{align*}
T_1 &\leqslant C_{f_0} Q(1-\ln Q)^2\\
\widetilde{T}_1 &\leqslant C_{\omega_0} \widetilde{Q}(1-\ln \widetilde{Q})^2\\
\end{align*}
Collecting all the previous bounds we obtain an estimate of the type, on $[0,t_\star]$
\begin{align*}
Q'+\widetilde{Q}'\leqslant C_{f_0,\omega_0}\big[Q+\widetilde{Q}-Q\ln Q-\widetilde{Q}\ln\widetilde{Q}\big],
\end{align*}
and by convexity we get finally
\begin{align*}
Q'+\widetilde{Q}'\leqslant C_{f_0,\omega_0}\big[Q+\widetilde{Q}-(Q+\widetilde{Q})\ln (Q+\widetilde{Q})\big],
\end{align*}
Since  $ Q(0) =  \widetilde{Q}(0) =0$, this allows to conclude that $Q$ and $\widetilde{Q}$ vanish on $[0,t_\star]$ and thus on $[0,T]$ by iterating the above argument.
\end{proof}
\section{Poisson structure }
\label{PoissonStructure}
The goal of this section is to exhibit the Hamiltonian structure of the equations \eqref{VE1}-\eqref{VE2},  following \cite{ArnoldKhesin}. 
We will be quite formal here, leaving aside the regularity issues. One therefore will think at some functions  $(\omega,f)$ which are smooth in $t,x,\xi$ with a nice decreasing at infinity.
First we endow the manifold $ \mathcal{P}$ of the pairs $(\omega,f)$ with a Poisson structure, that is to say, we endow ${\mathcal P}$ with a bracket $\{ \cdot  , \cdot \}$ acting on $\mathscr{C}^{\infty}$ functionals $F:{\mathcal P} \rightarrow \R$, bilinear and skew-symmetric,  satisfying the Jacobi and the Leibniz identities. Here this is obtained by setting, for any smooth functionals $F , G$ on $ \mathcal{P}$, 
\begin{equation*}
\{F , G \} := 
\int_{\R^2  \times  \R^2}  f \ \nabla_{\xi} F_f \cdot \nabla_{\xi}^\perp G_f 
+ \int_{\R^2  \times  \R^2} f \  \{F_f , G_f \}_{x,\xi}
   + \int_{\R^2   }  w \ \nabla_{x} F_w \cdot  \nabla_{x}^\perp G_w ,
\end{equation*}
where  $F_\omega$ and $ F_f$ denote the gradients  with respect to $ \omega $ and $f$ of a functional $F$ in $L^{2}$, given by the formula 
\begin{eqnarray*}
\int_{\R^2   } F_\omega (\omega,f) h = \lim_{\eps \rightarrow 0} \frac{F (\omega +  \eps h,f) - F (\omega ,f) }{\eps} \text{ and } \int_{\R^2  \times  \R^2  } F_f (\omega,f) \widetilde{h} = \lim_{\eps \rightarrow 0} \frac{F (\omega ,f+ \eps \widetilde{h}) - F (\omega ,f) }{\eps}
\end{eqnarray*}
and the bracket $\{ \cdot  , \cdot \}_{x,\xi}$ stands for
\begin{equation*}
\{ f  , g \}_{x,\xi} := \nabla_{x} \, f \cdot  \nabla_{\xi} \, g - \nabla_{\xi}\,  f \cdot  \nabla_{x}\,  g .
\end{equation*}
The properties cited above are clear from this definition. \par
Let us now define  $\mathcal{H}$ by 
\begin{eqnarray*}
2 \mathcal{H} =  \int_{\R^2 \times  \R^2 } \xi^{2} f(t,x,\xi) dx d\xi  
 -  \int_{\R^2 \times  \R^2}  G (x-y) (\omega +  \rho) (t,x) (\omega +  \rho) (t,y) dx dy  , \quad \text{ with } \rho :=  \int_{\R^2  } f d\xi ,
\end{eqnarray*}
 which can also be seen as a functional on the manifold $\mathcal{P}$, and this endows the system with a Hamiltonian structure in the following sense.
\begin{Proposition}\label{ProPoisson}
When $(  \omega,f )$ solves  the equations \eqref{VE1}-\eqref{VE2} we have for any smooth functional $F$, the ordinary differential equation
\begin{equation}
\label{haha}
\frac{d}{dt} F  = \{F , \mathcal{H} \},
\end{equation}
where $F$ and $ \mathcal{H}$ in \eqref{haha} stand respectively for $ F ( \omega,f ) $ and $   \mathcal{H}( \omega,f ) $.
\end{Proposition}
\begin{proof}
According to the chain rule, we have 
\begin{equation*}
\frac{d  }{dt} F (\omega,f )  =  \int_{\R^2  } \frac{\partial  \omega}{\partial t} F_\omega +  \int_{\R^2 \times  \R^2 } \frac{\partial  f}{\partial t} F_f  .
\end{equation*}
Using the equations \eqref{VE1}--\eqref{VE2} we  arrive to 
\begin{equation*}
\frac{d  }{dt}  F ( \omega , f) =  -   \int_{\R^2  } u  \cdot \nabla_{x} \omega \, F_\omega 
 -   \int_{\R^2 \times  \R^2 } \xi   \cdot  \nabla_{x}  f  \, F_f 
  -   \int_{\R^2 \times  \R^2 }  (\xi - u  )^{\perp} \cdot  \nabla_{\xi } f  \, F_f 
  =: I_1 + I_2 +I_3 .
\end{equation*}
On the other side the derivatives of $ \mathcal{H}$ are 
\begin{equation*}
\mathcal{H}_w = -  \int_{\R^2  }  G( \cdot - y  ) \, (\omega + \rho  ) (y) \, dy  , \quad \mathcal{H}_f =  \frac{\xi^{2}}{2} -  \int_{\R^2  }   G( \cdot - y  ) \, (\omega + \rho  ) (y) \, dy ,
\end{equation*}
Then, integrating by parts, we find 
\begin{equation*}
I_1 =   \int_{\R^2  } \omega  \nabla_{x}  \, F_\omega \cdot u   =  =   \int_{\R^2  } \omega  \nabla_x F_\omega  \cdot \nabla_x^\perp  \mathcal{H}_\omega ,
\end{equation*}
since 
$u = \nabla_x^\perp  \mathcal{H}_\omega $.

On the other side, since
$\xi = \nabla_\xi  \mathcal{H}_f $
we have
\begin{eqnarray*}
I_2 =   \int_{\R^2 \times  \R^2 } f   \nabla_{x} F_f   \cdot \nabla_{\xi }  \mathcal{H}_{f }  \text{ and }
 I_3 =   \int_{\R^2 \times  \R^2 } f \nabla_{ \xi } F_f  \cdot   \nabla_{ \xi }^\perp  \mathcal{H}_{f }  -  \int_{\R^2 \times  \R^2 } f   \nabla_{ x}  \mathcal{H}_{f }  \cdot   \nabla_{ \xi }  F_f  ,
\end{eqnarray*}
what yields  \eqref{haha}.
\end{proof}
As a simple corollary of \eqref{haha} and of the skew-symmetry of the bracket we get that $\mathcal{H}$ is conserved by the solutions of  \eqref{VE1}-\eqref{VE3}. \par
\section{Massless limit}
\label{MassLess}
In this section we investigate the behavior, when $\eps \rightarrow 0^+$ of the system 
\begin{eqnarray}
\label{VE1eps}
 \partial_t \omega^\eps + \div_x (\omega^\eps  u^\eps)  &=& 0 ,
\\   \label{VE2eps} 
 \partial_t  f^\eps + \div_x (f^\eps  \xi ) + \frac{1}{\eps} \div_\xi (  f^\eps (\xi - u^\eps  )^{\perp}) &=& 0 ,
\end{eqnarray}
where 
\begin{eqnarray}
\label{VE3eps} u^\eps :=K  \lbrack \omega^\eps + \rho^\eps  \rbrack  \text{ and } \rho^\eps :=  \int_{\R^2  } f^\eps d\xi .
\end{eqnarray}
This corresponds to the equations \eqref{VE1}-\eqref{VE2}-\eqref{VE3} with an extra factor $ {\eps}^{-1}$ which encodes a regime where the particles immersed into the fluid are lighter and lighter as  $\eps \rightarrow 0^+$.

We are going to prove that, in this limit, the equations  \eqref{VE1eps}-\eqref{VE2eps}-\eqref{VE3eps} degenerates into the incompressible Euler equation:
\begin{eqnarray}
\label{Euler}
 \partial_t u  + \div_x (u \otimes  u)  + \nabla p &=& 0  ,
\\ \label{EulerDIV} \div_x u  &=& 0. 
\end{eqnarray}
An heuristic way to guess the previous limit is the following: we infer from the equation \eqref{VE2eps} that in the limit $\eps \rightarrow 0^+$ the density of particles becomes monokinetic with a velocity $\xi =  u$ so that 
\begin{eqnarray}
 j^\eps := \int_{ \R^{2}}  f^\eps \xi d\xi 
\end{eqnarray}
converges to $ \rho u$,  where $ \rho$ and $ u$ stand for the respective limits of   $ \rho^\eps$ and $ u^\eps$.
Therefore the right hand side of the equation \eqref{VE1v} vanishes, again formally, what yields \eqref{Euler}-\eqref{EulerDIV}.

This problem is very close to the so-called gyrokinetic limit considered in \cite{brenier} with here an additional coupling to an Euler-type equation. 
  We will actually follow the strategy of \cite{brenier} based on the use of a modulated energy, with a few modifications. In particular we will make use of the notion of dissipative solutions of the Euler equations. This notion was  introduced by P.-L. Lions in  \cite{lions} and used in \cite{brenier}  with a slight modification in the definition.  Here we will perform another slight modification by extending Brenier's definition to the case of infinite energy, since arguably in $2$d the case of a incompressible perfect fluid of finite energy is too restrictive.
  For instance, it is easy to see from the definitions  \eqref{BS} and  \eqref{NBS} that, given  a smooth compactly supported function $g$ from $\R^{2}$ to $\R$, we have that  
  \begin{eqnarray}
\label{crazy}
 K [ g] \text{ is in  }   L^{2} \text{ if and only if  } \int_{\R^{2}} g(x) dx  = 0.
 \end{eqnarray}
  To deal with vorticities which do not satisfy this last condition, we follow  \cite{CheminSMF}  (see also \cite{Majda}) by defining the following.
  Let $\alpha \in \R$. There exists $g_{\alpha} \in \mathscr{D}(]0,\infty[; \R )$  such that $2 \pi  \displaystyle\int_{0}^{+\infty}  g_{\alpha} (r) r dr = \alpha $.
  We then define $H_{\alpha} := K [ g_{\alpha} (\|  \cdot \|) ]$ and  the space 
  $$E_{\alpha} :=  H_{\alpha} + L^{2}_{\sigma}   (\R^{2}), $$
  where  $ L^{2}_{\sigma}   (\R^{2})$ denotes the  divergence free vector fields in  $ L^{2}  (\R^{2})$.
  Observe that $H_{\alpha}$ is a smooth stationary solution of the the $2$d incompressible Euler equation \eqref{Euler}-\eqref{EulerDIV}.
  Finally the finite energy case corresponds to $\alpha = 0$ : thanks to \eqref{crazy}, there holds $E_{0} = L^{2}$.
\begin{Definition}[] 
\label{dissip}
 Let $\alpha \in \R$ and  $u_0 \in E_{\alpha} $. 
We say that a divergence free vector field $u \in L^{\infty} ((0,T) ; E_{\alpha} ) $ is a dissipative solution of the incompressible Euler equation \eqref{Euler}-\eqref{EulerDIV} with  $u_0$ as initial data if for any smooth vector field $v \in \mathscr{C}^0 ( [ 0,T] ; E_{\alpha} ) $, such as $A(v):=\partial_t v+ v\cdot \nabla v\in L^1((0,T);L^2(\R^2))$, and $\dd(v) \in L^1((0,T);L^2(\R^2))$ (where $\dd(v)$ is the symmetric part\footnote{This means $\dd(v)_{ij} := \frac{1}{2} (\partial_{j} v_{i} + \partial_{i} v_{j} ).$} of $D(v)$), for almost every $t \in [0,T ]$, 
\begin{eqnarray}
\label{WS}
\int_{\R^2} | u(t,x) - v(t,x)  |^{2} dx \leqslant \int_{\R^2} | u_{0} (x) - v(0,x)  |^{2} dx   \exp \int_0^{t} 2 \|  \dd(v(\theta))  \| d\theta 
\\ \nonumber + 2  \int_0^{t} \int_{\R^2} A(v)(s,x) (v-u)(s,x) \exp \left\{  \int_s^{t}2 \|  \dd(v(\theta))  \| d\theta\right\}  dx ds ,
\end{eqnarray}
where $ \|  \dd(v(\theta))  \|$ is the supremum in $x$ of the spectral radius of $\dd(v) (\theta,x)$.
\end{Definition}
\begin{Remark}
\label{rem:lions} Just as in \cite{lions}, let us note that by a standard regularization procedure it is sufficient to verify estimate \eqref{WS} for (only) all smooth (in time/space) vector fields $v$ satisfying $v(t)\in E_\alpha$ and $\curl v(t)$ is compactly supported, for all $t$.
\end{Remark}
Let us emphasize the two  modifications with respect to the Definition in  \cite{lions}: the first one, already done in  \cite{brenier}, is that we use the spectral radius of the whole matrix 
   $\dd(v) $, not only its negative part. The second one is that we deal with any $\alpha \in \R$, not only the case $\alpha=0$.
  
  The key property of this kind of solutions is that it allows the following weak-strong uniqueness result:
  if $\widetilde{u}$ is a smooth solution of the incompressible Euler equation \eqref{Euler}-\eqref{EulerDIV} with  $u_0$ as initial data then $u = \widetilde{u}$. 
  This can be easily obtained by observing that $A( \widetilde{u})$ (respectively $\widetilde{u} -u$) being a gradient (resp. divergence free) then  the integral $\displaystyle \int_{\R^2} A( \widetilde{u}) (\widetilde{u} -u) dx$ vanishes.
\begin{Theorem}[] 
\label{modulated}
Let be given  $\alpha \in \R$ and  $u_0 \in E_{\alpha} $. 
Let be given some smooth   solutions  $(\omega^\eps , f^\eps )_\eps$, with a nice decay of $f^\eps$ as $\xi$ goes to $\infty$, of the equations \eqref{VE1eps}-\eqref{VE2eps}-\eqref{VE3eps} corresponding to some smooth compactly supported initial data $(\omega^\eps_0 , f^\eps_0 )_\eps$ such that 
\begin{eqnarray}
\label{modulated00}
(\omega^\eps_0 )_\eps \text{ is bounded in } L^{2}  (\R^2 ) \text{ and } \left(\rho^\eps_{0} :=  \int_{\R^2  } f^\eps_{0} d\xi  \right)_\eps \text{ is bounded in } L^{1}  (\R^2 )
\end{eqnarray}
and such that, when $\eps \rightarrow 0^+$,
\begin{eqnarray}
\label{modulated01}
\eps \int_{\R^2 \times \R^{2}}  | \xi |^{2} f^{\eps}_{0} (x,\xi) dx d\xi  \rightarrow 0 , \quad  \int_{\R^2}  | u^\eps_0 - u_{0 }|^{2} dx  \rightarrow 0 ,
\end{eqnarray}
where $u^\eps_0 :=  K [\omega^\eps_0 + \rho^\eps_{0}]$.
Then, up to an extraction, the sequence $(u^\eps )_\eps$ converges in $L^{\infty} ((0,T) ; E_{\alpha} -w) $ to a dissipative solution  of the incompressible Euler equation with initial condition $u_{0}$.
\end{Theorem}
Note that as a simple corollary of Theorem \ref{modulated} and of the  weak-strong uniqueness result mentioned above we have in the particular case where   $u_0$ is smooth with bounded vorticity that the whole sequence $(u^\eps )_\eps$ converges  to the unique smooth solution  of the incompressible Euler equation with initial condition $u_{0}$.
Let us refer here again to  \cite{CheminSMF}  and  \cite{Majda} for the existence and uniqueness of smooth solutions of the incompressible  Euler equations in the plane with infinite energy.
\begin{proof}
Notice first that the sequence $(u^\eps)_\eps(t)$ takes values in $E_\alpha$, since 
\begin{align*}
\int_{\R^2} \omega^\eps(t,x) + \rho^\eps(t,x) dx = \int_{\R^2} \omega^\eps(0,x) + \rho^\eps(0,x) dx = \alpha.
\end{align*}
Consider a smooth (in time/space) vector field $v$ such as $v(t)\in E_\alpha$ and $\curl v(t)$ is compactly supported, for all $t$. Let us denote
\begin{eqnarray}
\label{defhv}
2 \mathcal{H}^{\eps}_{v} (t)  := \eps  \int_{\R^2 \times \R^2  } | \xi - v(t,x)  |^{2} f^{\eps}(t,x,\xi) dx d\xi + \int_{\R^2} | u^{\eps}(t,x) - v(t,x)  |^{2} dx .
\end{eqnarray}
The first thing to say about the modulated energy  $\mathcal{H}^{\eps}_{v} (t) $  is that it is the sum of two nonnegative finite terms.
In order  to explain in what sense it can be understood as a modulation of the energy 
 \begin{align}
 \label{lavraienergy}
  2 \mathcal{H}^\eps(t) :=  \eps M_2 (f^\eps)(t) - \int_{\R^2\times\R^2} (\rho^\eps+\omega^\eps)(t,x)(\rho^\eps+\omega^\eps)(t,y)G(x-y) dx dy,
 \end{align}
 one can reformulate it in the following way. Let us denote $h := \curl v$ so that  
\begin{eqnarray}
\label{bill}
v =K[h] =  \nabla^\perp (G\star h).
\end{eqnarray}
Then, using \eqref{bill} and \eqref{VE3eps},  we get 
\begin{eqnarray}
\label{square}
2 \mathcal{H}^{\eps}_{v}(t) = \eps  \int_{\R^2  \times \R^2} | \xi - v(t,x)  |^{2} f^{\eps}(t,x,\xi) dx d\xi   -\int_{\R^2\times\R^2} (\rho^\eps+\omega^\eps-h)(t,x)(\rho^\eps+\omega^\eps-h)(t,y)G(x-y) dx dy,
\end{eqnarray}
so that one recovers the energy  $\mathcal{H}^{\eps} (t) $ by setting $v=0$ in \eqref{square}.

It is also not difficult to deduce from the assumptions  \eqref{modulated00} and \eqref{modulated01} that  
\begin{eqnarray}
\label{inimodul}
\mathcal{H}^{\eps}_{v} (0)  \rightarrow \int_{\R^2} | u_{0} (x) - v(0,x)  |^{2} dx \text{ when } \eps  \rightarrow 0.
\end{eqnarray}
Moreover the dynamic of $\mathcal{H}^{\eps}_{v}$ is given by the following lemma:
\begin{Lemma}
\label{lengthy}
We have
\begin{eqnarray}
\label{lengthyCVRAI}
(\mathcal{H}^{\eps}_{v} )' (t) &=& - \eps  \int_{\R^2 \times \R^2 } \dd(v) (t,x): (\xi- v(t,x) ) \otimes  (\xi- v(t,x) )  f^{\eps}(t,x,\xi) dx d\xi 
\\ \nonumber &&+ \int_{\R^2  } \dd(v) (t,x): (u^{\eps} - v)^{\perp} (t,x) \otimes  (u^{\eps} - v)^{\perp} (t,x) dx 
\\ \nonumber  &&+ \eps \int_{\R^2  } A(v) (t,x) \cdot (\rho^{{\eps}} v - j^{\eps} ) dx 
\\ \nonumber  &&+  \int_{\R^2  } A(v) (t,x)  \cdot( v - u^{\eps}) dx .
\end{eqnarray}
\end{Lemma}
Recall the following classical notations, with the repeated index convention, 
\begin{enumerate}
\item $\xi\otimes\zeta$ is the matrix $(\xi_i \zeta_j)_{i,j}$ and, given $A$ a matrix, $\nabla : A $ is the vector $(\partial_{i} A_{ij})_{j}$.
\item If $A$ and $B$ are two matrices, $A:B$ is the scalar $A_{ij}B_{ij}$.
\end{enumerate}
\begin{proof}
To prove  this lemma we proceed as in the Appendix of  \cite{brenier}, with a few  modifications.
Let us observe that for any time,  $f^\eps (t)$, $\omega^\eps (t)$  have a compact support. 
Expanding the square of the right hand side of \eqref{square} we get 
\begin{eqnarray*}
2 \mathcal{H}^{\eps}_{v}(t) &=& \eps M_2(f^\eps)(t) + \eps \int_{\R^2}\rho^\eps |v|^2(t,x) dx - 2\eps  \int_{\R^2}  j^\eps(t) \cdot v(t)  dx 
\\ && -\int_{\R^2\times\R^2} (\rho^\eps+\omega^\eps-h)(t,x)(\rho^\eps+\omega^\eps-h)(t,y)G(x-y) dx dy,
\end{eqnarray*}
 hence
\begin{align*}
2\mathcal{H}^{\eps}_{v}(t)-2\mathcal{H}^\eps(t) &=  \eps \int_{\R^2}\rho^\eps |v|^2(t,x) dx - 2\eps  \int_{\R^2}  j^\eps(t) \cdot v(t)  dx -\int_{\R^2\times\R^2} h(t,x)h(t,y)G(x-y) dx dy\\
&+2\int_{\R^2\times\R^2} (\rho^\eps+\omega^\eps)(t,x)h(t,y)G(x-y) dx dy.
\end{align*}
Therefore  using that the energy $\mathcal{H}^\eps(t)$ is  independent of the time, we get 

\begin{align*}
2{\mathcal{H}^{\eps}_{v}}'(t) &= \,\stackrel{2I_1}{\overbrace{\eps \int_{\R^2} \partial_t \rho^\eps|v|^2 (t,x) d x}}\quad + \quad\stackrel{2I_2}{\overbrace{\eps \int_{\R^2} \rho^\eps \partial_t |v|^2(t,x) d x}}\quad\stackrel{2I_3}{\overbrace{-2\eps  \int_{\R^2}   \partial_t j^\eps(t) \cdot v(t)  dx}} \quad\stackrel{2I_4}{\overbrace{-2\eps   \int_{\R^2}j^\eps(t) \cdot  \partial_t v(t) dx }} \\
&\stackrel{2I_5}{\overbrace{-2\int_{\R^2\times\R^2} \partial_t h(t,x)h(t,y)G(x-y) dx dy}}\\
&+\quad\stackrel{2I_6}{\overbrace{2\int_{\R^2\times\R^2} (\rho^\eps+\omega^\eps)(t,x)\partial_t h(t,y)G(x-y) dx dy}} \quad + \quad\stackrel{2I_7}{\overbrace{2\int_{\R^2\times\R^2} \partial_t(\rho^\eps+\omega^\eps)(t,x)h(t,y)G(x-y) dx dy}}.
\end{align*}
We are going to deal first with $I_7$ and then with $I_3$.
By integrating \eqref{VE2eps} with respect to $\xi$ we get 
\begin{align}
\label{The1}
\partial_t \rho^\eps + \text{div}_x(j^\eps) &= 0.
\end{align}
Adding this to \eqref{VE1eps} then yields
\begin{align*}
\partial_t(\rho^\eps+\omega^\eps) + \textnormal{div}_x[j^\eps +\omega^\eps u^\eps] = 0,
\end{align*}
We therefore have
\begin{align*}
I_7 =  \int_{\R^2} \nabla (G\star h) \cdot [j^\eps+\omega^\eps u^\eps] d x
=  \int_{\R^2} v\cdot [j^\eps+\omega^\eps u^\eps]^\perp  d x ,
\end{align*}
by using \eqref{bill}.

Let us also observe that multiplying  \eqref{VE2eps} by $\xi$ and then  integrating with respect to $\xi$ implies the following  vectorial equation
 \begin{align}
\label{The2}
\partial_t j^\eps + \nabla : \left\{\int_{\R^2} \xi\otimes\xi f^\eps d\xi\right\} &= \frac{1}{\eps}(j^\eps-\rho^\eps u^\eps)^\perp.
\end{align}
Using this we get also
\begin{align*}
I_3 &= \eps \int_{\R^2} v\cdot \left[\nabla : \int_{\R^2} \xi \otimes \xi f^\eps d\xi\right]d x + \int_{\R^2} \rho^\eps v \cdot {u^\eps}^\perp dx - \int_{\R^2} v\cdot {j^\eps}^\perp dx\\
&=- \eps \int_{\R^2} \dd(v) : \left\{\int_{\R^2} \xi\otimes \xi f^\eps d\xi\right\} dx + \int_{\R^2}(\rho^\eps +\omega^\eps)v\cdot {u^\eps}^\perp dx - I_7.
\end{align*}
And since $\nabla\Big[G\star(\rho^\eps+\omega^\eps)\Big]=-{u^\eps}^\perp$ which divergence gives back $\rho^\eps+\omega^\eps$, we have
\begin{align*}
\int_{\R^2}(\rho^\eps +\omega^\eps)v\cdot {u^\eps}^\perp dx &= \int_{\R^2}\Delta_{x}\Big[G\star(\rho^\eps +\omega^\eps)\Big]v\cdot {u^\eps}^\perp dx\\
&= \int_{\R^2}{u^\eps}^\perp \cdot \nabla (v\cdot {u^\eps}^\perp) dx\\
&= \int_{\R^2}{u^\eps}^\perp\otimes {u^\eps}^\perp : \nabla v\, dx  + \int_{\R^2}{u^\eps}^\perp\cdot \nabla {u^\eps}^\perp \cdot v \,dx.
\end{align*}
where the second integral vanishes since $2 {u^\eps}^\perp \cdot \nabla {u^\eps}^\perp= \nabla \Big[G\star(\rho^\eps+\omega^\eps)^2\Big]$ and $v$ is divergence free. 
We eventually get by symmetry
\begin{align*}
I_3+I_7 = - \eps \int_{\R^2} \dd(v) : \left\{\int_{\R^2} \xi\otimes \xi f d\xi\right\} dx +\int_{\R^2} \dd(v) : {u^\eps}^\perp\otimes {u^\eps}^\perp d x,
\end{align*}
that we decompose into:
\begin{align*}
I_3+I_7 &= \stackrel{Q_1}{\overbrace{-\eps \int_{\R^2} \dd(v) :\left\{\int_{\R^2}(\xi-v)\otimes(\xi-v)f^\eps d\xi\right\}dx + \int_{\R^2} \dd(v): (v^\perp-{u^\eps}^\perp)\otimes (v^\perp-{u^\eps}^\perp) dx}}\\
&+\stackrel{Q_2}{\overbrace{-\eps\int_{\R^2} \nabla v : j^\eps\otimes v \, d x}}+\stackrel{Q_3}{\overbrace{-\eps\int_{\R^2} \nabla v : v \otimes j^\eps \, d x}}+\stackrel{Q_4}{\overbrace{\int_{\R^2} \nabla v : v^\perp\otimes {u^\eps}^\perp \, d x}}+\stackrel{Q_5}{\overbrace{\int_{\R^2} \nabla v : {u^\eps}^\perp\otimes v^\perp \, d x}}\\
&+\stackrel{Q_6}{\overbrace{\eps\int_{\R^2} \rho \nabla v : v \otimes v \, d x}}+\stackrel{Q_7}{\overbrace{-\int_{\R^2} \nabla v : v^\perp \otimes v^\perp \, d x}}.
\end{align*}
Therefore we have 
\begin{eqnarray*}
{\mathcal{H}^{\eps}_{v}}'(t)  = I_1 + I_2 + I_4
+ I_5 + I_6 + Q_1 + ...  + Q_7 .
\end{eqnarray*}
Let us see a first cancellation: we have, with the repeated index convention,
\begin{align*}
Q_3 = -\eps \int_{\R^2} (\partial_j v_i)\,v_i\,j^\eps_j dx
=\frac{\eps}{2} \int_{\R^2}  |v|^2 \textnormal{div}_x(j^\eps) dx
=-I_1 ,
\end{align*}
by using \eqref{The1}.

Thus 
\begin{eqnarray*}
{\mathcal{H}^{\eps}_{v}}'(t)  &=&
  I_2 + I_4
+ I_5 + I_6 + Q_1 + Q_2 +
Q_4 + \ldots + Q_7   
\\ &=& Q_1
 + (I_2 + I_4 + Q_2 + Q_6 )
+ (I_5 + I_6 + Q_4 + Q_5+ Q_7 ) .
\end{eqnarray*}
Let us observe that $Q_1$ is equal to the sum of the first two terms of the right hand side of 
\eqref{lengthyCVRAI}, so that, in order to prove Lemma \ref{lengthy}, it is sufficient to prove that the third term is equal to $I_2 + I_4 + Q_2 + Q_6$ and that the fourth term  is equal to $I_5 + I_6 + Q_4 + Q_5+ Q_7 $ (here we have distinguished the terms with a factor $\eps$ and the other ones).

Regarding the third term, we have 
\begin{align*}
\int_{\R^2} A(v)\cdot (\rho^\eps v-j^\eps) dx &= \frac{1}{2}\int_{\R^2} \rho^\eps \partial_t |v|^2 dx - \int_{\R^2} \partial_t v \cdot j^\eps dx + \int_{\R^2} v_j (\partial_j v_i) \rho^\eps v_i dx -  \int_{\R^2} v_j (\partial_j v_i) j^\eps_i dx \\
&= \frac{1}{2}\int_{\R^2} \rho^\eps \partial_t |v|^2 dx - \int_{\R^2} \partial_t v \cdot j^\eps dx + \int_{\R^2} \rho^\eps \nabla v : v\otimes v\, dx -  \int_{\R^2} \nabla v : j^\eps \otimes v\, dx,
\end{align*}
and hence, multiplying by $\eps$, we get:
\begin{align*}
 \eps \int_{\R^2} A(v)\cdot (\rho^\eps v-j^\eps) dx =  I_2+I_4+Q_2+Q_6 .
\end{align*}

We finally have (we use $\text{div}_x v= 0$ for the first equality)
\begin{align*}
\int_{\R^2} A(v) \cdot (v-u^\eps)dx &= \frac{1}{2}\int_{\R^2} \partial_t |v|^2 dx - \int_{\R^2} \partial_t v \cdot u^\eps dx -  \int_{\R^2} v_j (\partial_j v_i) u^\eps_i dx \\
&=I_5+I_6 - \int_{\R^2} \nabla v : u^\eps \otimes v\,dx ,
\end{align*}
so that it only remains to prove that
\begin{align*}
 - \int_{\R^2} \nabla v : u^\eps \otimes v\,dx =Q_4 + Q_5+ Q_7 .
\end{align*}
Actually we are going to prove that 
\begin{align*}
  \int_{\R^2} \nabla v : u^\eps \otimes v\,dx + Q_4= 0 \text{ and }  Q_5 + Q_7 = 0  .
\end{align*}
In order to do so, let us first note that we have the following relations for $a,b$ and $c$ three vector fields
\begin{align*}
\Big[c\cdot \nabla a\Big]^\perp = c\cdot \nabla a^\perp 
\quad  \Rightarrow \quad 
 \nabla a : b\otimes c = \Big[c\cdot \nabla a\Big] \cdot b =  \nabla a^\perp : b^\perp\otimes c.
\end{align*}
Hence
\begin{eqnarray*}
  \int_{\R^2} \nabla v : u^\eps \otimes v\,dx  + Q_4 &=&\int_{\R^2} \nabla v : \Big[u^\eps \otimes v + v^\perp\otimes {u^\eps}^\perp \Big]\, dx
  \\ &=& \int_{\R^2} \nabla v^\perp : \Big[{u^\eps}^\perp \otimes v - v\otimes {u^\eps}^\perp \Big]\, dx,
\end{eqnarray*}
and since $v=K[h]=\nabla^\perp (G\star h)$, we get 
\begin{align*}
  \int_{\R^2} \nabla v : u^\eps \otimes v\,dx  + Q_4  &= -\int_{\R^2} \partial_j \partial_i (G\star h) \Big[{u^\eps_i}^\perp v_j-v_i {u^\eps_j}^\perp\Big] dx = 0,
\end{align*}
by symmetry.

On the other hand
\begin{align*}
 Q_5 + Q_7 = &=\int_{\R^2} \nabla v : \Big[(u^\eps-v)^\perp \otimes v^\perp\Big] dx \\
 &= \int_{\R^2} \nabla v^\perp : \Big[(u^\eps-v) \otimes v^\perp\Big] dx \\
&= \int_{\R^2} \partial_j \partial_i (G\star h) (u^\eps_i-v_i) \partial_j(G\star h) dx \\
&= \int_{\R^2} \partial_i \partial_j (G\star h) (u^\eps_i-v_i) \partial_j(G\star h) dx\\
&= \int_{\R^2} \nabla \left|\frac{1}{2}\nabla (G\star h)\right|^2 \cdot (u^\eps-v)dx =0.
\end{align*}
\end{proof}

With Lemma \ref{lengthy} in hands the proof of Theorem \ref{modulated} then follows the proof of Theorem $6.1$ in  \cite{brenier}, with again a few modifications. 
For instance a crucial estimate in the strategy of \cite{brenier} is that $\|  \sqrt{\eps} j^\eps \|_{L^\infty_t (L^1_x)}$ is bounded uniformly in $\eps$.
Unlike Brenier, we cannot use here the energy of the system as the second term in  \eqref{lavraienergy} does not have a definite sign. 
Instead we are going to bootstrap an estimate involving both $\|  \sqrt{\eps} j^\eps \|_{L^\infty_t (L^1_x)}$ and $ \| \mathcal{H}^{\eps}_{v} \|_{L^\infty }$.

By Cauchy-Schwarz inequality we get 
\begin{eqnarray}
\nonumber
\sqrt{\eps}  \int_{\R^2}|j^\eps|(t,x) dx &\leq& \sqrt{\eps}  \int_{\R^2\times\R^2} |\xi-v(t,x)| f^\eps(t,x,\xi) d\xi\, dx + \sqrt{\eps} \int_{\R^2} |v(t,x)| \rho^\eps(t,x) \,dx \\
    \label{jsisi}                                                   &\leq&  \Big( (\mathcal{H}^{\eps}_{v}(t) )^{\frac{1}{2}} +  \| v  \|_{L^\infty}    \Big) \|  \rho^\eps_{0} \|_{L^1} .
\end{eqnarray}
Above we used that \eqref{The1} implies that  $  \|  \rho^\eps (t) \|_{L^1} = \|  \rho^\eps_{0} \|_{L^1}$ for any $t$.

Now, using  lemma \ref{lengthy}, we deduce that  
\begin{align*}
{\mathcal{H}^{\eps}_{v}}'(t) \leq  2 \|\dd(v)\|_{L^\infty}\mathcal{H}^{\eps}_{v}(t) + \eps \|A(v)\|_{L^\infty}\|v\|_{L^\infty}\|\rho^{\eps}_{0}\|_{L^1} + \eps \|A(v)\|_{L^\infty}\|j^\eps\|_{L^\infty_t(L^1_x)} + \sqrt{2} \|A(v)\|_{L^\infty_{t}(L^2_x)} (\mathcal{H}^{\eps}_{v}(t))^{\frac{1}{2}},
\end{align*}
that is,
\begin{align*}
{\mathcal{H}^{\eps}_{v}}'(t) \leq \textnormal{P}_{\text{in}}\Big(\|v\|_{W^{1,\infty}}+\|v\|_{L^\infty_t(W^{1,2}_x)}\Big)\big[1+\mathcal{H}^{\eps}_{v}(t)\big],
\end{align*}
for some polynomial function $\textnormal{P}_{\text{in}}$. Here we used the assumption \eqref{modulated00} to bound $\|\rho^{\eps}_{0}\|_{L^1}$.

A Gronwall lemma gives then that for a fixed smooth vector field $v$ (such as $v(t)\in E_\alpha$ for all $t$),
\begin{eqnarray}
\label{bornehv}
\mathcal{H}^{\eps}_{v} \text{ is bounded in}\, L^\infty_t ,
\end{eqnarray}
which implies, according to \eqref{defhv} and \eqref{jsisi}, that 
\begin{align}
\label{bornej}
(\sqrt{\eps} j^\eps)_\eps\quad&\text{is bounded in}\quad L^\infty_t(L^1_x)\\
\label{borneu} (u^\eps-v)_\eps \quad&\text{is bounded in}\quad L^\infty_t(L^2_x).
\end{align}

Let us now look for a bound of $\rho^\eps {u^\eps}$. 
We start by writing
\begin{align}
\label{compo}
\rho^\eps {u^\eps} = (\rho^\eps+\omega^\eps-h)\stackrel{K[\rho^\eps+\omega^\eps-h]}{\overbrace{({u^\eps}-v)}} + (h-\omega^\eps)(u^\eps-v) +\rho^\eps v.
\end{align}
From \eqref{VE1eps} we infer that $  \|  \omega^\eps (t) \|_{L^2} = \|  \omega^\eps_{0} \|_{L^2}$ for any $t$. Using the assumption \eqref{modulated00} we deduce that 
$(\omega^\eps)_\eps$ is bounded in $L^\infty_t(L^2_x)$, and we recall that  $(\rho^\eps)_\eps$ is bounded in $L^\infty_t(L^1_x)$.
Since $h,v\in L^\infty_{t,x}$, using \eqref{borneu} we see  that the two last factors of the right hand side are bounded in $ L^\infty_t(L^1_x)$. 

For the first one we will use, 
as in \cite{brenier},  the following lemma.
\begin{Lemma}
For any $g$ smooth such as $\displaystyle\int_{\R^2}g(x)dx = 0$, we have 
\begin{eqnarray}
\| gK[g] \|_{W^{-1,1}}\leq \frac{3}{2}\|K[g]\|_{L^2} .
\end{eqnarray}
\end{Lemma}
Let us provide a proof for sake of completeness.
\begin{proof}
Let  $\varphi$ be a smooth compactly supported vector field over $\R^2$. 
We have
\begin{align*}
\int_{\R^2} \varphi(x)\cdot  g(x) K[g]^\perp(x) dx 
&=- \int_{\R^2} \varphi \cdot \Delta\Big[G\star g\Big]\nabla \Big[G\star g\Big]dx\\
&=- \int_{\R^2} \varphi_i \partial_{jj}\Big[G\star g\Big]\partial_i \Big[G\star g\Big]dx \\
&= \int_{\R^2} \varphi_i \partial_{j}\Big[G\star g\Big]\partial_j \partial_i \Big[G\star g\Big]dx + \int_{\R^2} \partial_j \varphi_i \partial_{j}\Big[G\star g\Big] \partial_i \Big[G\star g\Big]dx \\
&= -\int_{\R^2} \Big( \text{div}(\varphi) (x) \frac{1}{2}\Big|\nabla G\star g\Big|^2+ \nabla\Big[G\star g\Big]\cdot \nabla(\varphi) \cdot \nabla\Big[G\star g\Big] \Big) dx,
\end{align*}
and hence 
\begin{align*}
\left|\int_{\R^2} \varphi(x)\cdot  gK[g](x) dx \right| \leq \frac{3}{2}\|\varphi\|_{W^{1,\infty}} \|K[g]\|_{L^2},
\end{align*}
what yields the result.
\end{proof}
The lemma above thus gives that the first term in the right hand side of \eqref{compo} is bounded in $L^\infty_t(W^{-1,1}_x)$, so that $\rho^\eps u^\eps$ is eventually bounded in $L^\infty_t(H^{-m}_x)$, for some $m$ large enough.

Now we are going to use this to deduce that $(\rho^\eps)_\eps$ is relatively compact in $\mathscr{C}^0_t(\mathscr{D}'(\R^2))$.
We first combine  \eqref{The1} and \eqref{The2} to get 
\begin{align*}
\partial_t (\rho^\eps - \eps \text{div}_x({j^\eps}^\perp)) = \eps \left[ \nabla : \left\{\int_{\R^2} \xi\otimes\xi f^\eps d\xi\right\}
\right]^\perp -\rho^\eps u^\eps.
\end{align*}
Because of  \eqref{bornehv} , the first term of the right hand side is bounded in $\mathscr{C}^0_t(W^{-1,1}_x)$ that we already injected in some $\mathscr{C}^0_t(H^{-m}_x)$ in which the second term is also bounded. Since $H^{-m}(B(0,N))\hookrightarrow H^{-m-1}(B(0,N))$ is compact for all $N\in \N$, Ascoli's theorem and a diagonal extraction allow us to conclude that $(\rho^\eps-\eps \text{div}_x({j^\eps}^\perp))_\eps$ is (up to an extraction) converging in $\mathscr{C}^0_t(\mathscr{D}'(\R^2))$. But $(\sqrt{\eps}j^\eps)_\eps$ is bounded in $\mathscr{C}^0_t(L^1_x)\hookrightarrow \mathscr{C}^0_t(\mathscr{D}'(\R^2))$ and hence we get the convergence of $(\rho^\eps)_\eps$ in $\mathscr{C}^0_t(\mathscr{D}'(\R^2))$  and we denote by $\rho$ its limit.

The previous convergence holds (up to an extraction) as well for $(\omega^\eps)_\eps\rightarrow \omega$, the reasoning being simplified because $(\omega^\eps)_\eps$ is bounded in every $L^\infty_t(L^p_x)$. We hence get a subsequence of $(u^\eps)_\eps=(K[\rho^\eps+\omega^\eps])_\eps$ that converges in $\mathscr{C}^0_t(\mathscr{D}'(\R^2))$ to some $u$. But $(u^\eps-v)_\eps$ is bounded in $\mathscr{C}^0_t(L^2_x)$ so that the previous convergence implies $(u^\eps-v)_\eps \rightarrow (u-v)$ in $\mathscr{C}^0_t(L^2_x-w)$, with furthermore the pointwise bound
\begin{align*}
\|u(t)-v(t)\|_{L^2} \leq \operatorname*{\underline{\lim}}_{\eps \rightarrow 0} \|u_\eps(t)-v(t)\|_{L^2}.
\end{align*}
$(\mathcal{H}^{\eps}_{v})_\eps$ is up to an extraction converging in $L^\infty_t- w \star$ to some function $\mathcal{H}_{v}$ and the previous bound gives 
\begin{align}
\label{ineq:enermod}\|u(t)-v(t)\|_{L^2}^2 \leq \mathcal{H}_{v}(t),
\end{align}
almost for all $t\in[0,T]$. Now lemma \ref{lengthy} gives, 
\begin{align*}
 {\mathcal{H}^{\eps}_{v}}'(t) &\leq 2 \|\dd(v)\|(t) \mathcal{H}^{\eps}_{v}(t) +  \int_{\R^2} A(v) (t,x)\cdot ( v - u^{\eps})(t,x)dx
+ \Big[1+ \|v\|_{W^{1,\infty}}^3\Big]\Big[\eps \|j^\eps\|_{L^\infty_t(L^1_x)}  + \eps  \|\rho^{{\eps}} \|_{L^\infty_t(L^1_x)}\Big],
\end{align*}
hence
\begin{align*}
 \mathcal{H}^{\eps}_{v}(t) &\leq \mathcal{H}^{\eps}_{v}(0) \exp\left\{ \int_0^t2 \|\dd(v)\|(s)\,ds\right\} +  \int_0^t \exp\left\{ \int_s^t 2 \|\dd(v)\|(\sigma)\,d\sigma\right\}\int_{\R^2} A(v) (s,x)\cdot ( v - u^{\eps})(s,x)dx\,ds
\\ &+ \Big[1+ \|v\|_{W^{1,\infty}}^3\Big]\Big[\eps \|j^\eps\|_{L^\infty_t(L^1_x)}  + \eps  \|\rho^{{\eps}} \|_{L^\infty_t(L^1_x)}\Big],
\end{align*}
Using previous bounds, we see that the product of the second line goes to $0$ with $\eps$. 
We  now pass to the limit in the previous inequality, using \eqref{inimodul}, to obtain
\begin{align*}
 \mathcal{H}_{v}(t) &\leq \|u_{0}-v\|_{L^2}^2 \exp\left\{ \int_0^t2 \|\dd(v)\|(s)\,ds\right\} +  \int_0^t \exp\left\{ \int_s^t 2 \|\dd(v)\|(\sigma)\,d\sigma\right\}\int_{\R^2} A(v) (s,x)\cdot ( v - u)(s,x)dx\,ds,
\end{align*}
first in the distributional sense (that is against any positive test function) and then almost everywhere since both sides of the previous inequality are $L^\infty$ functions.

Then \eqref{ineq:enermod} allows to obtain the dissipation formulation introduced in Definition \ref{dissip}, at least for any smooth (in time/space) vector field $v$ such as $v(t) \in E_\alpha$ and  compactly supported $\curl v(t)$, for all $t$, which is sufficient in view of Remark \ref{rem:lions}.
\end{proof}
\bigskip
{\bf Acknowledgements.} The second author was partially supported by the Agence Nationale de la Recherche, Project CISIFS,  grant ANR-09-BLAN-0213-02.  
He thanks Olivier Glass and Walter Strauss for some fruitful discussions.


\begin{thebibliography}{10}
\bibitem{criline}
     C.~Lacave, E.~Miot. 
     Uniqueness for the vortex-wave system when the vorticity is constant near the point vortex.
 SIAM J. Math. Anal. 41(3), p1138-1163, 2009.
\end{thebibliography}
\end{document}